\documentclass{amsart}
\usepackage{cite,amsmath,amssymb,amsthm,ulem}
\usepackage{graphicx}
\usepackage{multirow}
\usepackage{subfigure}
\usepackage{cleveref}
\usepackage[linesnumbered, ruled, vlined]{algorithm2e}

\newtheorem{theorem}{Theorem}[section]
\newtheorem{lemma}{Lemma}[section]

\newtheorem{remark}{Remark}[section]

\numberwithin{figure}{section}
\numberwithin{table}{section}

\def\be{\begin{equation}}
\def\ee{\end{equation}}
\def\bes{\begin{equation*}}
\def\ees{\end{equation*}}
\def\beq{\begin{eqnarray}}
\def\eeq{\end{eqnarray}}
\def\beqs{\begin{eqnarray*}}
\def\eeqs{\end{eqnarray*}}
\def\bal{\begin{aligned}}
\def\eal{\end{aligned}}
\def\bsqs{\begin{subequations}}
\def\esqs{\end{subequations}}

\begin{document}
\title[A HDDG method for Poisson--Nernst--Planck systems]{A positivity-preserving hybrid DDG method for Poisson--Nernst--Planck systems
}
\author[H.~Liu, Z.-M. Wang and P.~Yin]{Hailiang Liu$^\dagger$, Zhongming Wang$^\ddagger$  and Peimeng Yin$^*$\\  \\
 }
\address{$^\dagger$ Mathematics Department, Iowa State University, Ames, IA 50011, USA} \email{hliu@iastate.edu}
\address{$^\ddagger$  Department of Mathematics and Statistics, Florida International University,  Miami, FL 33199, USA}\email{zwang6@fiu.edu}
\address{$^*$ Department of Mathematical Sciences, The University of Texas at El Paso, El Paso, TX 79968, USA}
\email{pyin@utep.edu}

\subjclass{35K40, 65M60, 65M12, 82C31.}
\keywords{Poisson-Nernst-Planck equation, positivity preserving, discontinuous Galerkin methods}
\begin{abstract} 
In earlier work [H. Liu and Z. Wang, J. Comput. Phys., 328(2017)], an arbitrary high-order conservative and energy-dissipative direct discontinuous Galerkin (DDG) scheme was developed. Although this scheme enforced solution positivity using cell averages as reference values, it lacked a theoretical guarantee for the positivity of those cell averages.  
In this study, we develop a novel arbitrary high-order DDG method with rigorously proven positivity-preserving properties. 
Specifically, the positivity of the cell averages is ensured through a modified numerical flux in combination with forward Euler time discretization. To achieve point-wise  positivity of ion concentrations, we introduce a hybrid  algorithm that integrates a positivity-preserving limiter. The proposed method is further extended to higher-dimensional problems with rectangular meshes.   Numerical results confirm the scheme's high-order  accuracy, guaranteed  positivity preservation, and consistent discrete energy dissipation.
\end{abstract}
\maketitle

\section{Introduction} 

This work continues our project, initiated in \cite{LW17}, aimed at developing a rigorous structure-preserving discontinuous Galerkin (DG) method for the 
Poisson--Nernst--Planck equations. In this paper, we focus on the PNP system given by 
\begin{subequations}\label{PNP}
\begin{align}
&\partial_t c_i= \nabla\cdot(\nabla c_i+q_ic_i\nabla\psi), \quad  x\in  \Omega, \; t>0, \quad i=1,\cdots, m,\\
&-\Delta \psi =  \sum_{i=1} ^m q_i c_i + \rho_0,  \quad x\in  \Omega, \; t>0, \\
&c_i(0,x)=c_i^{\rm in}(x), \quad x\in \Omega;  \quad \frac{\partial c_i}{\partial  \textbf{n}}+q_ic_i \frac{\partial \psi}{\partial  \textbf{n}} =0  \mbox{~~on~}  \partial\Omega, \; t>0,\label{PNPc} \\
& \psi =\psi_D \mbox{~~on~} \partial\Omega_D,  \mbox{~~and~} \frac{\partial \psi}{\partial  \textbf{n}}  =\sigma  \mbox{~~on~} \partial\Omega_N, \quad t>0,  \label{PNPd}
\end{align}
\end{subequations}
where $c_i=c_i(t, x) $ represents the concentration of the $i^{th}$ ionic species with charge $q_i$,
and is subject to the given 
initial and boundary conditions. The  electrostatic potential $\psi=\psi(t, x)$ is governed by the Poisson equation, with Neumann boundary conditions prescribed on $\partial\Omega_N$ and Dirichlet boundary 
conditions on $\partial\Omega_D$.  Here,  $\Omega \subset \mathbb{R}^d$ denotes a connected closed domain with a smooth boundary $\partial\Omega$,  $\textbf{n}$ is the unit outward normal vector on $\partial \Omega$.  

The PNP system (\ref{PNP}) is widely used across many disciplines \cite{Mo83, Ma86, MRS90, Je96, Da97, Li04, Gl42, Hi01, EL07}.
Among the key qualitative properties of its solution -- such as mass conservation, energy dissipation, and positivity of concentrations--  the preservation of positivity under high-order numerical approximations 
remains particularly challenging.

In this article,  we advance our previous work by focusing on the challenge of positivity preservation within the same  high-order DG approximation framework introduced in \cite{LW17}. 
A general discussion of this issue and background references are given in the introduction to \cite{LW17}.   Recall that the direct DG method for \eqref{PNP} proposed in \cite{LW17} 
is  based on the following reformulation: 
\begin{align*}
 \partial_t c_i=\nabla \cdot (c_i \nabla p_i),  \quad  p_i =q_i \psi +\log c_i, \;  \; i=1, \cdots, m, 
\end{align*}
where the electrostatic potential $\psi$ is determined from the coupled Poisson equation. This reformulation, when discretized using DG methods, was shown to naturally preserve both mass conservation and energy dissipation. 
A critical ingredient  of the method is the use of the direct DG (DDG) numerical fluxes on $p_i$, where the DDG fluxes are originally introduced in  \cite{LY09,LY10}),  of the form for $w=p_i$:   
\begin{align}\label{flux_intro}
\widehat{\partial_n w}:=\beta_0 \frac{[w]}{h} +\{\partial_n w\}+\beta_1h[\partial_n^2 w].
\end{align} 
Here  $n$ denotes the normal unit vector across cell interfaces, and $[q]=q^+-q^-$, and $\{q\}={(q^++q^-)}/{2}$, where $q^+$ and $q^-$ representing values of $q$ on either side of the interface. A broad range of parameter choices 
 for $\beta_0$ and $\beta_1$ allows for the scheme to retain high-order accuracy while ensuring energy dissipation. However, this flexibility does not generally guarantee the preservation of  solution positivity, which is not unexpected, achieving pointwise positivity in high-order schemes is well known to be unrealistic.
 
 A commonly accepted strategy, following Shu and Zhang \cite{ZS10}, is to ensure the positivity of cell averages and then apply  a positivity-preserving limiter to enforce solution positivity. Unfortunately, the numerical method proposed in \cite{LW17} may produce negative cell averages after a finite number of time steps,  thereby causing the positivity-preserving limiter to fail.
 
In this work, we aim to overcome this critical limitation identified in \cite{LW17}, by developing a scheme that theoretically guarantees positivity of cell averages, even in general settings.  
We have two main objectives in this work: 

i) To present a novel modified DDG numerical flux, and  

ii) To analyze the propagation of positive cell averages. 

The goal of (i) is to restore the positivity propagation of cell averages through a carefully  designed flux modification. 
Our modified flux is based on a local correction constructed similarly to those given for one-dimensional nonlinear Fokker-Planck equations \cite{CLY25}.
However, in our case, the correction is applied to  the logarithmic flux formulation 
$p_i=q_i \psi+\log c_i$.  Specifically, we define the modified normal derivative as:   
\begin{equation}
    \widetilde{\partial_n p_i}=\widehat{\partial_n p_i} + \frac{\widetilde\beta}{2}[c_i].
\label{mflux_intro}
\end{equation}

Here $ \tilde \beta$ is dimension-dependent and chosen to appropriately scale the correction.
 
For (ii), we analyze the evolution of cell averages under Euler time discretization.  In the one dimensional case, the update for the cell average 
 $\bar c_{ij}$ can be expressed as  
$$
\bar c_{ij}^{n+1} = \bar c_{ij}^n + \frac{\Delta t}{\Delta x} \{c_{ih}^n\} \left(\widehat{\partial_x p_{ih}^n} + \frac{\widetilde{\beta}_{i}}{2}[c_{ih}^n]\right)\Big|^{x_{j+\frac12}}_{x_{j-\frac12}},
$$
where the flux terms are evaluated at cell interfaces.  When using the modified flux (\ref{mflux_intro}) and applying the Gauss-Labotto decomposition of  $\bar c_{ij}^n$, we show that all the interface trace terms  
stay non-negative,  provided the time step is sufficiently small.  This leads to a key result: the propagation of positive cell averages -- that is, $\bar c_{ij}^{n+1} > 0$ if the current DG polynomials $c_{ih}^{n} > 0$ and the time step satisfies a suitable 
constraint. 

As a result,  we develop a hybrid DDG numerical algorithm that incorporates a positivity-preserving limiter,  similar to the one proposed in \cite{LY15},  
applied at each time step. This limiter enforces the positivity of the numerical solution without degrading the scheme's high-order accuracy.  Our analysis and numerical results demonstrate that the resulting positivity-preserving DDG scheme maintains arbitrary high-order accuracy in both one-dimensional and multi-dimensional settings.
\subsection*{Related work.} 
The development of high-order, positivity-preserving numerical methods for the PNP system remains relatively limited. A notable contribution  is  found in \cite{LWYY22}, where  the authors proposed a direct DG method that rigorously preserves the 
positivity of cell averages.  Their method relies on the use of the non-logarithmic reformulation and carefully constructed numerical fluxes to enforce positivity; however, the use of global fluxes limits the achievable order of accuracy, 
 making extensions to higher-order DG methods challenging.  In contrast, the local flux strategy developed  in this work, as well as in  \cite{CLY25},
 overcomes this limitation  and  attains spatial accuracy beyond third order.   

Similar ideas of using local flux corrections to enforce desirable solution properties, particularly positivity preservation, have been explored in other settings, such as \cite{Zh17} for the compressible Navier-Stokes equations. 
 
After our earlier work \cite{LW17},  which introduced a novel framework for achieving arbitrary high-order spatial accuracy while preserving key structural properties of the solution, there has been a growing body of research 
focused on designing numerical schemes that maintain essential solution features at the discrete level. These efforts typically involve reformulating the original PDE and adopting implicit-explicit time discretizations to balance stability and computational efficiency. 
Most of such schemes use central differencing for spatial discretization, typically resulting in second-order spatial accuracy. 

Common reformulation strategies include:  

(1) Non-logarithmic transformations,  referred to as the Slotboom transformation in semiconductor modeling -- convert the drift-diffusion operator into a self-adjoint elliptic operator.  This is particularly effective for ensuring positivity; see e.g., \cite{DingWangZhou_JCP2019,HPY19, DingWangZhou_JCP2020,HH2020,LM20b, LM21, LWYY22, ST22}. 

(2) Gradient flow structures of the form $\rho_t =\nabla \cdot(\rho \nabla \delta_\rho  E)$, where $E$ denotes an energy functional associated 
with the physical system, are often discretized using explicit-implicit schemes;  examples include the energetic variational approach (EnVarA) \cite{LiuC2021a,LWWYZ2023} and an independent development in  \cite{ShenXu_NM21}. In \cite{MXL16}, the authors propose a fully-implicit time discretization combined with a log-density formulation  to ensure positivity.  
 
(3) Localized gradient flow formulations,  such as  $\rho_t =\nabla \cdot(\rho \nabla p)$ with $p= \delta_\rho E$ (see \cite{LW17}) or $\rho_t =\nabla \cdot(\rho \xi )$  with $\xi = \nabla \delta_\rho E$ (see  \cite{CCH15,SCS18}), provide enhanced control over the solution structure over computational cells;  

(4) Fully implicit gradient flows, often framed as JKO-type schemes \cite{JKO98, KLX17},  pose significant computational challenges. These are effectively addressed by recasting them as dynamic optimization problems \cite{LM23}, following the Benamou-Brenier formulation \cite{BB00}. Notably, such dynamic optimal transport formulations can also be naturally derived using the Onsager variational principle \cite{CLX25}. 
These perspectives provide promising avenues for developing robust, structure-preserving numerical schemes.

We conclude this section by outlining the rest of the paper:  in Section 2, we review the DDG method introduced in \cite{LW17} and summarize its key 
solution properties, including mass conservation and the energy dissipation law.  Section 3 focuses on the positivity-preserving property of the fully discretized scheme, incorporating the proposed novel local flux correction  and the hybrid positivity-preserving algorithm in one dimension. In Section 4,  we extend the method to higher dimensions and prove the positivity propagation results. 
Section 5 presents numerical examples that illustrate the effectiveness of the proposed approach. Finally, Section 6 offers 
concluding remarks.

\section{Review of Direct DG Scheme in \cite{LW17}}
 In this section, we briefly revisit the DG scheme for the  PNP system (\ref{PNP}) in 1D as presented in \cite{LW17}.
Consider a domain $\Omega = [a,  b]$ with a possibly non-uniform mesh $\{I_j\}_{j=1}^N$, where the cell $I_j =(x_{j-1/2}, x_{j+1/2})$  with cell center
$x_j =(x_{j-1/2} +x_{j+1/2})/2$, and 
$$
a=x_{1/2}<x_1<\cdots <x_{N-1/2}<x_N<x_{N+1/2}=b.
$$

We define the discontinuous finite element space as
$$
V_h=\left\{v\in L^2(\Omega), \quad v|_{I_j}\in P^k(I_j), j=1, \cdots, N \right\},
$$
where $P^k$ denotes polynomials of degree at most $k$, and $h$ the characteristic length of the mesh. 
At cell interfaces $x=x_{j+1/2}$, we introduce the notation
\begin{equation*}
    v^{\pm} = \lim_{\epsilon \rightarrow 0}v(x\pm \epsilon), \quad \{v\} = \frac{v^-+v^+}{2}, \quad [v] = v^+- v^-.
\end{equation*}
The DDG scheme introduced in \cite{LW17} is essentially based  on the following localized gradient flow reformulation,  
\begin{subequations}\label{cqp}
\begin{align}
 \partial_t c_i&=\partial_x (c_i \partial_x p_i),  \; i=1, \cdots, m, \\
 p_i &=q_i \psi +\log c_i,\\
 -\partial_{x}^2\psi&=\sum_{i=1}^m q_ic_i +\rho_0(x).
\end{align}
\end{subequations}
More specifically,  such a scheme when coupled with the   first-order Euler time discretization with  uniform step size  $\Delta t$  admits the following form: We seek 
$c_{ih}^{n+1}, \psi^{n+1}_h \in V_h$ such that for any $v_i, r_i, \eta \in V_h$,
\begin{subequations}\label{dgEuler}
\begin{align}
& \int_{I_j} \frac{c_{ih}^{n+1}-c_{ih}^n}{\Delta t}  v_i dx =-\int_{I_j} c_{ih}^n \partial_x p_{ih}^n \partial_x v_i dx +\{c_{ih}^n\} \left( \widehat{\partial_x p_{ih}^n}v_i + (p_{ih}^n-\{p_{ih}^n\})
\partial_x v_{i}\right)\Big|_{\partial I_j}, \label{dgEulera}\\
 & \int_{I_j}p_{ih}^n r_i dx =\int_{I_j}(q_i \psi_h^n + \log c_{ih}^n )r_{i}dx,\\
 & \int_{I_j}  \partial_x \psi_{h}^n \partial_x \eta  dx - \left( \widehat { \partial_x \psi_{h}^n} \eta   +(\psi_h^n - \{\psi_h^n\})\partial_x \eta \right)\Big|_{\partial I_j}  = \int_{I_j}
  \left[\sum_{i=1}^m q_i c_{ih}^n +\rho_0\right]
 \eta dx,
\end{align}
\end{subequations}
where the numerical fluxes are $\widehat{\partial_x p_{ih}}=Fl(p_{ih})$ and $\widehat{ \partial_x\psi_{h}} =Fl(\psi_h)$. The flux operator $Fl$ at the cell interfaces $x_{j+1/2}$ for $0<j<N$ is defined as 
\begin{align}\label{flux}
Fl(w):=\beta_0 \frac{[w]}{h_{j+\frac12}} +\{\partial_xw\}+\beta_1h_{j+\frac12}[\partial_x^2 w],
\end{align}
where
$$
h_{j+\frac12}=x_{j+1}-x_j. 
$$

The boundary conditions are a critical component of the PNP model
and determine important qualitative behavior of the solution.  It is important to incorporate the boundary conditions properly. For the zero-flux condition $\partial_x c_i +q_i c_i \partial_x \psi=0$ at both $x=a$ and $x=b$, we set at both $x_{1/2}$ and $x_{N+1/2}$,
\begin{align}
& Fl(p_{ih})=0, \{p_{ih}\}=p_{ih},
\{c_{ih}\}=c_{ih}. 
\end{align} 
For $\psi$ with mixed boundary condition of form 
$\psi(t, a)=\psi_l, \partial_x\psi(t, b)=\sigma$, we set
\begin{align}
& Fl(\psi_{h})=\beta_0 (\psi_{h}^+ -\psi_l)/h_{1/2} + \partial_x \psi_{h}^+, \{\psi_{h}\}=\psi_{l}, \quad x_{1/2}=a,\\
& Fl(\psi_{h})=\sigma, \{\psi_{h}\}=\psi_{h}^-, \quad x_{N+1/2}=b, 
\end{align}
where 
$$
h_{1/2}=x_1-x_{1/2}. 
$$
Such a choice is motivated by the results in \cite{Liu21}; otherwise, one would have to take $2\beta_0$ to replace the usual coefficient $\beta_0$.
For other types of boundary conditions, numerical fluxes can be defined accordingly.

As usual, we initialize $c_{ih}^{{0}}$ by projecting $c_{i}^{\rm in}$ onto $V_h$ so that
\begin{equation}\label{proj}
\int_{I_j} c_{ih}^{0}(x) v(x)dx=\int_{I_j} c_{i}^{\rm in}(x)v(x)dx, \quad \forall v\in P^k(I_j), \quad i=1, \cdots, m.
\end{equation}

The following statement holds for the DG scheme \eqref{dgEuler} when subjected to a Neumann boundary condition.
\begin{theorem} {\t cf. \cite[Theorem 3.3]{LW17}}\label{thmEuler} 
\begin{itemize} 
\item[1.]  The fully discrete scheme (\ref{dgEuler}) conserves each  total concentration $c_{ih} ^n(x)(i=1, \cdots,  m)$ over time: 
\begin{align}
\label{conservediscrete}
 \sum_{j=1}^N \int_{I_j} c_{ih}^n  dx =   \sum_{j=1}^N \int_{I_j} c_{ih} ^{n+1} dx, \quad i=1,\cdots, m, \quad \quad t>0.
\end{align}
\item[2.] Assuming $c_{ih}^n(x)>0$ in each $I_j$, there exists $\mu^*>0$ such that if the mesh ratio $\mu=\frac{\Delta t}{\Delta x^2} \in (0, \mu^*)$, then the fully discrete free energy
\be 
\label{fullyDiscreteEnergy}
F^n =\sum_{j=1}^N  \int_{I_j}  \left[ \sum_{i=1}^m c_{ih} ^n{\rm log} c_{ih}^n+\frac{1}{2} \left( \sum_{i=1}^m q_i c_{ih}^n+\rho_0 \right) \psi_h^n \right] dx +\frac{1}{2}\int_{\partial \Omega}\sigma  \psi_h^n ds
\ee
satisfies
 \begin{align}
D_tF^n&\leq  - \frac{1}{2} \sum_{i=1}^m A_{c_{ih}^n}(p_{ih}^n, p_{ih}^n),
\label{fdEuler}
\end{align}
where the bilinear form 
$$
A_{c_{ih}^n}(p_{ih}^n, v_i) = \sum_{j=1}^N \int_{I_j} c_{ih}^n \partial_x p_{ih}^n \partial_x v_i dx + \sum_{j=0}^N \{c_{ih}^n\} \left(  \widehat{\partial_x p_{ih}^n}[v_i] + [p_{ih}^n]
\{\partial_x v_{i} \}\right)_{j+\frac{1}{2}}.
$$
Moreover,
 \begin{align}\label{ff}
F^{n+1}\leq F^n,
\end{align}
provided that $\beta_0$ is suitably large, and $\beta_1=0$ in $Fl(\psi_h)$  defined in (\ref{flux}).
\end{itemize}
\end{theorem}
A significant limitation in the above results is the assumption $c_{ih}^n>0$. It is well known that maintaining point-wise positivity directly through high order numerical approximations is unrealistic. A widely accepted approach, following Shu and Zhang \cite{ZS10}, is to ensure positivity of cell averages, and restore solution positivity by further applying  a positivity-preserving limiter. However, the numerical scheme \eqref{dgEuler} presented in \cite{LW17} can  produce negative cell averages in finite number of steps, ultimately causing the positivity-preserving limiter to fail.  This is confirmed in \Cref{1Dtestacc} and \Cref{2Dtestacc} of this paper for both 1D and 2D cases. Our main objective here is to introduce a remedy to ensure the positivity of numerical averages.

\section{A positivity-preserving flux}
The key ingredient is a locally modified flux $\widetilde{\partial_xp_{ih}}$ which is used to replace the original flux $\widehat{\partial_x p_{ih}^n}$ in \eqref{dgEulera} at each edge $\partial I_j$. The resulting DDG scheme becomes 
\begin{subequations}\label{fullyPDG}
\begin{align}
& \int_{I_j} \frac{c_{ih}^{n+1}-c_{ih}^n}{\Delta t} v_i dx =-\int_{I_j} c_{ih}^n \partial_x p_{ih}^n \partial_x v_i dx +\{c_{ih}^n\} \left(  \widetilde{\partial_x p_{ih}^n}v_i + (p_{ih}^n-\{p_{ih}^n\})
\partial_x v_{i}\right)\Big|_{\partial I_j}, \label{fullyPDGa} \\
 & \int_{I_j}p_{ih}^n r_i dx =\int_{I_j}(q_i \psi_h^n + \log c_{ih}^n )r_{i}dx, \label{fullyPDGb}\\ 
 & \int_{I_j}  \partial_x \psi_{h}^n \partial_x \eta  dx - \left( \widehat { \partial_x \psi_{h}^n} \eta   +(\psi_h^n - \{\psi_h^n\})\partial_x \eta \right)\Big|_{\partial I_j}  = \int_{I_j}
  \left[\sum_{i=1}^m q_i c_{ih}^n +\rho_0\right] 
 \eta dx, \label{fullyPDGc}
\end{align}
\end{subequations}
where the newly introduced flux at the cell interface $x_{j+\frac12}$ for $0<j<N$ is given by
\begin{equation}
\widetilde{\partial_xp_{ih}}=\widehat{\partial_x p_{ih}} + \frac{\widetilde{\beta}_{i}}{2}[c_{ih}],
\label{Pflux}
\end{equation}
and 
\begin{equation}\label{defbeta12}
    \widetilde{\beta}_{i}= \left\{\begin{array}{cl}
        \frac{|\widehat{\partial_x p_{ih}^n}|}{\{c_{ih}^n\}} \Big|_{x_{j+\frac12}}, &\text{ if } \{c_{ih}^n\} > 0; \\
0, &\text{ if } \{c_{ih}^n\} = 0.
\end{array}\right.
\end{equation}
For $j=0$ or $j=N$, no modification is imposed so that  
$$
\widetilde{\partial_xp_{ih}}=\widehat{\partial_x p_{ih}}=0, \quad \tilde \beta_i=0.
$$

\subsection{Positive  cell averages} 
We define the cell average on cell $I_j$ as 
$$
\bar{c}_{ij}^{n} := \frac{1}{\Delta x_j}\int_{I_j}c_{ih}^{n}dx,\quad \Delta x_j={x_{j+\frac12}-x_{j-\frac12}}, 
$$
and establish the following result: 

\begin{theorem}\label{THMfullyPDG}
For the fully discrete scheme \eqref{fullyPDG}, the cell average $\bar c_{ij}^{n+1} > 0$ provided $c_{ih}^n(x)>0$ and 

\begin{align}\label{CFL1D}
& \Delta t \leq \omega_1 \min_{1\leq j\leq N}
\left\{
\Delta x_j \left.|\widehat{\partial_x p_{ih}^n}|^{-1}\right|_{x_{j\pm\frac12}}
\right\}
\end{align}
is satisfied. Here, $\omega_1$ represents the first weight of the Gauss-Lobatto quadrature rules with $M \geq \frac{k+3}{2}$ points. For notational convenience, we include the terms for $j=0, N$ with $\left.|\widehat{\partial_x p_{ih}^n}|^{-1}\right|_{x_{j+1/2}} = \infty$.
\end{theorem}
\begin{proof}
We define
\begin{align*}
\lambda_j :=\frac{\Delta t}{(\Delta x_j)^2}.
\end{align*}
Taking $v_i = \frac{\Delta t}{\Delta x_j}$ in \eqref{fullyPDGa}, the time evolution of $\bar c_{ij}^{n+1}$ on $I_j$ is given by
\begin{align*}
    \bar c_{ij}^{n+1} =& \bar c_{ij}^n + \lambda_j\{c_{ih}^n\} \Delta x_j \left(\widehat{\partial_x p_{ih}^n} + \frac{\widetilde{\beta}_{i}}{2}[c_{ih}^n]\right)\Big|_{x_{j+\frac12}}
    -\lambda_j \Delta x_j \{c_{ih}^n\}\left(\widehat{\partial_x p_{ih}^n} + \frac{\widetilde{\beta}_{i}}{2}[c_{ih}^n]\right)\Big|_{x_{j-\frac12}}\\
    =& \sum_{m=1}^M \omega_mc_{ih}^n(x_m^*) + \frac{\lambda_j}{2}\Delta x_j \left(\widehat{\partial_x p_{ih}^n} + \widetilde{\beta}_{i}\{c_{ih}^n\}\right)\Big|_{x_{j+\frac12}}c_{ih}^n(x_{j+\frac12}^+) \\
    & \hspace{2.3cm} +\frac{\lambda_j}{2} \Delta x_j \left(\widehat{\partial_xp_{ih}^n} - \widetilde{\beta}_{i}\{c_{ih}^n\}\right)\Big|_{x_{j+\frac12}}c_{ih}^n(x_{j+\frac12}^-)\\
&\hspace{2.3cm}-\frac{\lambda_j}{2} \Delta x_j \left(\widehat{\partial_x p_{ih}^n} + \widetilde{\beta}_{i}\{c_{ih}^n\}\right)\Big|_{x_{j-\frac12}} c_{ih}^n(x_{j-\frac12}^+) \\
&\hspace{2.3cm} +\frac{\lambda_j}{2} \Delta x_j \left(-\widehat{\partial_xp_{ih}^n} + \widetilde{\beta}_{i}\{c_{ih}^n\}\right)\Big|_{x_{j-\frac12}}c_{ih}^n(x_{j-\frac12}^-),
\end{align*}
where the $M$ pairs $(\omega_m, x_m^*)$ are given by the Gauss-Lobatto quadrature rules with $M \geq \frac{k+3}{2}$ that are exact for $k$-th polynomials and
\[
x_1^* = x_{j-\frac12}, \quad x_M^* = x_{j+\frac12}.
\]
Thanks to the choice of $\widetilde{\beta}_{i}$ in \eqref{defbeta12},
 we have
\[
    \left(\widehat{\partial_x p_{ih}^n} + \widetilde{\beta}_{i}\{c_{ih}^n\}\right) \Big|_{x_{j+\frac12}} \geq 0, \quad \left(-\widehat{\partial_x p_{ih}^n} + \widetilde{\beta}_{i}\{c_{ih}^n\}\right)\Big|_{x_{j-\frac12}} \geq 0,
\] 
for $1\leq j\leq N$. 

It follows that
\begin{align*}
\bar c_{ij}^{n+1} =& \sum_{m=2}^{M-1} \omega_m c_{ih}^n(x_m^*)
+ \frac{\lambda_j}{2}\Delta x_j \left( \widehat{\partial_xp_{ih}^n} + \widetilde{\beta}_{i}\{c_{ih}^n\}\right)c_{ih}^n(x_{j+\frac12}^+)\\
&\hspace{2.6cm}
+ \frac{\lambda_j}{2} \Delta x_j \left(-\widehat{\partial_xp_{ih}^n} + \widetilde{\beta}_{i}\{c_{ih}^n\}\right)c_{ih}^n(x_{j-\frac12}^-)\\
&\hspace{2.6cm}+ \left[\omega_1 - \frac{\lambda_j}{2} \Delta x_j (\widehat{\partial_x p_{ih}^n} + \widetilde{\beta}_{i}\{c_{ih}^n\})\right]c_{ih}^n(x_{j-\frac12}^+)\\
&\hspace{2.6cm}+ \left[\omega_M - \frac{\lambda_j}{2} \Delta x_j (-\widehat{\partial_x p_{ih}^n} + \widetilde{\beta}_{i}\{c_{ih}^n\})\right]c_{ih}^n(x_{j+\frac12}^-)>0,
\end{align*}
provided
\[
    \lambda_j \leq \min \left\{\frac{\omega_1}{\Delta x_j} \left.|\widehat{\partial_x p_{ih}^n}|^{-1}\right|_{x_{j-\frac12}},   \frac{\omega_1}{\Delta x_j} \left.|\widehat{\partial_x p_{ih}^n}|^{-1}\right|_{x_{j+\frac12}} \right\}
\]
for $2\leq j \leq N-1$, and 
\begin{align*}
  &  \lambda_1 \leq \frac{\omega_1}{\Delta x_1} |\widehat{\partial_x p_{ih}^n}|^{-1}\Big|_{x_{\frac32}} = \min \left\{\frac{\omega_1}{\Delta x_1} \left.|\widehat{\partial_x p_{ih}^n}|^{-1}\right|_{x_{\frac12}},   \frac{\omega_1}{\Delta x_1} \left.|\widehat{\partial_x p_{ih}^n}|^{-1}\right|_{x_{\frac{3}{2}}} \right\},\\
    &  \lambda_N \leq \frac{\omega_1}{\Delta x_N} |\widehat{\partial_x p_{ih}^n}|^{-1}\Big|_{x_{N-\frac12}} = \min \left\{\frac{\omega_1}{\Delta x_N} \left.|\widehat{\partial_x p_{ih}^n}|^{-1}\right|_{x_{N-\frac12}},   \frac{\omega_1}{\Delta x_N} \left.|\widehat{\partial_x p_{ih}^n}|^{-1}\right|_{x_{N+\frac12}} \right\},
\end{align*} 
where the equalities follow from the fact that
$$
\left.|\widehat{\partial_x p_{ih}^n}|^{-1}\right|_{x_{1/2}} = \left.|\widehat{\partial_x p_{ih}^n}|^{-1}\right|_{x_{N+1/2}} = \infty.
$$
Note that $\omega_1 = \omega_M$ due to the symmetry of the Gauss-Lobatto quadrature rules. 
\end{proof}

\subsection{Positivity-preserving reconstruction}  \label{recon}
The equation \eqref{fullyPDGb} involves the term \( \log c_{ih}^n \), which requires the concentrations \( c_{ih}^n \) to remain positive at each time step. However,   \Cref{THMfullyPDG} only guarantees positive cell averages each step, provided the previous  profile is positive. To ensure the propagation of positivity over time, we apply an accuracy-preserving  positive limiter  based on positive cell averages.

Let \( w_h \in P^k(I_j) \) approximate a smooth function \( w(x) \geq 0 \), with cell averages \( \bar{w}_j > \delta \), where \( \delta \) is a small positive parameter or zero. We define the modified polynomial   \( w_h^\delta(x) \in P^k(I_j) \) as follows:  

\begin{equation}\label{limiter}
w_h^\delta(x) =
\begin{cases}
\bar{w}_j + \frac{\bar{w}_j - \delta}{\bar{w}_j - \min_{I_j} w_h(x)} (w_h(x) - \bar{w}_j), & \text{if } \min_{I_j} w_h(x) < \delta, \\
w_h(x), & \text{otherwise}.
\end{cases}
\end{equation}
This reconstruction preserves the original cell averages and ensures 
that 
\[
\min_{I_j} w_h^\delta(x) \geq \delta.
\]
More importantly, the following estimate demonstrates how the choice of $\delta$  affects the reconstruction  accuracy. 
\begin{lemma} {\t cf. \cite[Lemma 4.2]{LW17}}  
 If \( \bar{w}_j > \delta \), the reconstructed polynomial \( w_h^\delta \) satisfies the following estimate  
\[
|w_h^\delta(x) - w_h(x)| \leq C(k) \big( \| w_h(x) - w(x) \|_\infty + \delta \big), \quad \forall x \in I_j,
\]
where \( C(k) \) is a constant depending on \( k \). 
\end{lemma}  
This result guarantees that \( w_h^\delta(x) \) maintains accuracy when \( \delta < h^{k+1} \).   
Such reconstruction limiting techniques are inspired by the limiter introduced in \cite{ZS10} for hyperbolic conservation laws.

\subsection{The hybrid positivity-preserving algorithm}

The algorithm is summarized as follows:

\normalem
    \begin{algorithm}[ht]
        \caption{The hybrid positivity-preserving algorithm for the PNP system}
        \label{alg:HDDG}
    \KwIn{Initial data \(c_i^{\rm in}(x)\), and the boundary condition for $\psi$}
    \KwOut{Concentration $c_{ih}^n$} 
        Initialization for \(c_{ih}^0\):  
    Project \(c_i^{\rm in}(x)\) onto \(V_h\), as formulated in (\ref{proj}), to obtain \(c_{ih}^0(x)\).\\
        \For{$n=0$ \KwTo $\lfloor T/\Delta t \rfloor$}{
            Reconstruction:  
        Ensure positivity in each cell by reconstructing \(c_{ih}^n\) following \eqref{limiter} if necessary, such that \(c_{ih}^n > \delta\).\\
            Solve for \(\psi_h^n\):  
        Using \(c_{ih}^n\), solve (\ref{fullyPDGc}) to compute \(\psi_h^n\).\\
        Solve for \(p_{ih}^n\):
        Using $c_{ih}^{n}$ and $\psi_{h}^{n}$ to obtain $p_{ih}^n$ from \eqref{fullyPDGb}.\\
        Solve for \(c_{ih}^{n+1}\):  
        Solve (\ref{fullyPDGa}) to obtain \(c_{ih}^{n+1}\).
        }
    \end{algorithm}

\begin{remark}\label{rem:CFL}
We provide some remarks concerning the numerical implementation.
\begin{itemize}
\item[1.]  
We assume the mesh is regular, satisfying $h_{\max} \leq Ch_{\min}$ for some constant $C$, where $h_{\max}$ and $ h_{\min}$ denote the maximum and minimum element diameters, respectively. Let $h$ be the characteristic length of the mesh size, 
then we have  $\widehat{\partial_x p_{ih}^n} \approx \mathcal{O}(h^{-1})$. 
From the CFL condition \eqref{CFL1D}, we observe that
\begin{equation}
    \Delta t \leq \mathcal{O}(h^2),
\end{equation}
which is consistent with the expected condition on the explicit time step for diffusion problems.

\item[2.] Due to the presence of a logarithmic term in the DG scheme \eqref{fullyPDGb}, we additionally require that the cell average
\be
\bar c_{ij}^{n+1} \geq \varepsilon, \notag
\ee
for some tolerance $\varepsilon \in (0, C_\eta h^{k+1})$, where $C_\eta$ is a constant. As a result, the time step is chosen numerically to satisfy
\begin{equation}\label{CFL1D+}
    \Delta t = \gamma\lambda h^2,
\end{equation}
where $\lambda = \frac{\omega_1}{h} \min_{0\leq j\leq N}
\left\{|\widehat{\partial_x p_{ih}^n}|^{-1}\Big|_{x_{j+1/2}}
\right\}$ , and $\gamma \in (0,1]$ is a empirical parameter,  which we investigate numerically in \Cref{sec:2d}. 
\item[3.] The time discretization in the scheme \eqref{fullyPDG} can be implemented using high order strong stability preserving Runge-Kutta schemes \cite{GST2001}, as they are positive linear combination of the forward Euler method used in \eqref{fullyPDG}, hence positive-preserving property remains valid when taking suitably small time steps.

\end{itemize}

\end{remark}

\begin{remark}
Although not formally proven, our extensive 1D and 2D numerical experiments in  \Cref{sec:num} show that the numerical solution produced from Algorithm \ref{alg:HDDG} exhibits the property of discrete mass conservation and free energy dissipation.   
\end{remark}

\section{High-dimensional positivity-preserving DG schemes}

In this section, we extend our result to DG schemes of $(k+1)$-th order accuracy on rectangular meshes solving multidimensional (multi-D) PNP equations. 

\subsection{Scheme formulation}
Assume that \eqref{PNP} is posed on $x=(x^1, \ldots, x^d) \in \Omega = \Pi_{l=1}^d [L_{x^l}, R_{x^l}] \subset \mathbb{R}^d$, where $L_{x^l}$, $R_{x^l}$ are endpoints in $x^l$ direction satisfying $L_{x^l} < R_{x^l}$.
Consider a rectangular partition $\mathcal{T}_h$ of 
\begin{equation}
    \Omega = \bigcup_{\alpha}^N K_\alpha, 
\end{equation}
where $\alpha=(\alpha_1, \ldots, \alpha_d)$, $N=(N_1, \ldots, N_d)$. Here, $K_\alpha = I_{\alpha_1}^1 \times \ldots \times I_{\alpha_d}^d$ with $I_{\alpha_l}^l = [x^l_{\alpha_l-1/2}, x^l_{\alpha_l+1/2}]$ and its center $x^l_{\alpha_l} = (x^l_{\alpha_l-1/2}+x^l_{\alpha_l+1/2})/2$, for $\alpha_l = 1, \ldots, N_l$.

We define the discontinuous finite element space as the space of the tensor product of piecewise polynomials of degree at most $k$ in each variable on every element, 
\[
V_h = \{v \in L^2(\Omega),\  v|_{K_\alpha} \in Q^k(K_\alpha), \ \forall \alpha=1, \ldots, N\},
\]
where $Q^k$ denotes the tensor producto of $P^k$ polynomials in each direction.
Similar to the one-dimensional case, we introduce the following notation at cell interfaces $x^l = x^l_{\alpha_l + 1/2}$,
\begin{equation*}
v^\pm = \lim_{\epsilon \rightarrow 0} v(x^1,\ldots, x^l\pm \epsilon, \ldots, x^d), \quad  \{v\} = \frac{v^-+v^+}{2}, \quad [v] = v^+- v^-.
\end{equation*}

The fully discrete DG scheme on each computational cell is to find $c_{ih}^{n+1}, \psi^{n+1}_h \in V_h$ such that for any $v_i, r_i, \eta \in V_h$,
\begin{subequations}\label{fullyPDGMD}
\begin{align}
\int_{K_\alpha}D_t c_{ih}^{n}v_i\,dx
=&  -\int_{K_\alpha} c_{ih}^n\nabla p_{ih}^n\cdot\nabla v_i\,dx \notag\\
& \hspace{-1cm}   + \sum_{l=1}^d \int_{K_\alpha \backslash I_{\alpha_l}^l}  \{c_{ih}^n\} \left(\widetilde{\partial_{x^l} p_{ih}^n}v_i +(p_{ih}^n -
\{p_{ih}^n\})\partial_{x^l} v_i\right)\,\Big|_{x^{l,+}_{\alpha_l-1/2}}^{x^{l,-}_{\alpha_l+1/2}}d(x\backslash x^l),\label{fullyMD1} \\
\int_{K_\alpha}p_{ih} r_i\,dx = & \int_{K_\alpha}\left( q_{i} \phi_h^n + \log c_{ih}^n  \right)r_i\,dx, \label{fullyMD2} \\
\int_{K_\alpha} \nabla \psi_{h}^n\cdot\nabla \eta\,dx 
& - \sum_{l=1}^d \int_{K_\alpha \backslash I_{\alpha_l}^l} 
\left(\widehat{\partial_{x^l} \psi_{h}^n}\eta +(\psi_{h}^n -\{\psi_{h}^n\})\partial_{x^l} \eta\right)\,\Big|_{x^{l,+}_{\alpha_l-1/2}}^{x^{l,-}_{\alpha_l+1/2}}d(x\backslash x^l) \nonumber\\
= & \int_{K_\alpha}  \left( \sum_{i=1}^m q_i c_{ih}^n + \rho_0  \right) \eta  \,dx, \label{fullyMD3}
\end{align}
\end{subequations}
where at interior interface $(x^1, \ldots, x^l_{\alpha_l+\frac12}, \ldots, x^d)$, 
the DDG numerical fluxes are given by
\begin{align}
    \widehat{\partial_{x^l} w} = \beta_0 \frac{[w]}{h^{l}_{\alpha_l+\frac12}} +\{\partial_{x^l} w\}+\beta_1h^{l}_{\alpha^l+\frac12}[\partial_{x^l}^2 w],
\end{align}
for
\begin{align*}
   h^{l}_{\alpha_l+\frac12} =  {x}^l_{\alpha_l+1}  - {x}^l_{\alpha_l};
\end{align*}
the modified numerical fluxes are defined as
\begin{align}\label{betaMD}
\widetilde{\partial_{x^l} p_{ih}^n}& =\widehat{\partial_{x^l} p_{ih}^n}+ \frac{\widetilde{\beta_{i}^l}}{2}[c_{ih}^n],
\end{align}
with the functions
\begin{equation}\label{defbetaM}
\widetilde{\beta_{i}^l}= \left\{\begin{array}{cl}
        \frac{|\widehat{\partial_{x^l} p_{ih}^n}|}{\{c_{ih}^n\}}, &\text{ if } \{c_{ih}^n\} > 0, \\
0, &\text{ if } \{c_{ih}^n\} = 0.
\end{array}\right.
\end{equation}
At the boundary $x^l = x_{\frac{1}{2}}^l$ or $x^l = x_{N_l+\frac{1}{2}}^l$, our guiding principle for defining numerical fluxes and averages is as follows: whenever exact boundary data is prescribed, it is used directly; otherwise, the interior numerical approximation is employed.
To enforce the zero-flux boundary condition in \eqref{PNPc}, we set 
\begin{equation*}
    \widehat{\partial_{x^l} p_{ih}^n} =0, \ \{p_{ih}^n\} = p_{ih}^n, \ \{c_{ih}^n\} = c_{ih}^n.
\end{equation*}
The modified numerical flux takes the given data, namely,
\begin{equation*}
    \widetilde{\partial_{x^l} p_{ih}^n} = \widehat{\partial_{x^l} p_{ih}^n} = 0, \quad \widetilde{\beta_{i}^l} =0.
\end{equation*}
For the potential $\psi$ subject to mixed boundary condition as in \eqref{PNPd}, we specify the numerical fluxes and the averages at the interface $x^l = x_{\frac{1}{2}}^l$ or $x^l = x_{N_l+\frac{1}{2}}^l$ on a case-by-case basis, depending on whether Dirichlet or Neumann conditions are imposed.
If a Dirichlet boundary condition is prescribed at the boundary, then:
\begin{align*}
& \widehat{\partial_{x^l} \psi_{h}^n} = \beta_0 \frac{\psi_{h}^{n,+} - \psi_D}{h^l_{\frac{1}{2}}} + \partial_{x^l} \psi_{h}^{n,+}, \ \{\psi_{h}^n\} = \psi_D, \quad \text{if } (x^1, \ldots, x^l_{\frac{1}{2}}, \ldots, x^d) \in \partial \Omega_D, \\
& \widehat{\partial_{x^l} \psi_{h}^n} = \beta_0 \frac{\psi_D-\psi_{h}^{n,-}}{h^l_{N_l+\frac{1}{2}}} + \partial_{x^l} \psi_{h}^{n,-}, \ \{\psi_{h}^n\} = \psi_D, \quad \text{if } (x^1, \ldots, x^l_{N_l+\frac{1}{2}}, \ldots, x^d) \in \partial \Omega_D,
\end{align*}
where 
\begin{equation*}
h^l_{1/2} = x_{1}^l - x_{1/2}^l, \quad h^l_{N_l+1/2} = x_{N_l+1/2}^l - x_{N_l}^l.
\end{equation*}
 
If a Neumann boundary condition is prescribed, then:
\begin{align*}
& \widehat{\partial_{x^l} \psi_{h}^n} =\sigma, \ \{\psi_{h}^n\} = \psi_{h}^{n,+}, \quad \text{if } (x^1, \ldots, x^l_{\frac{1}{2}}, \ldots, x^d) \in \partial \Omega_N, \\
& \widehat{\partial_{x^l} \psi_{h}^n} = \sigma, \ \{\psi_{h}^n\} = \psi_{h}^{n,-}, \quad \text{if } (x^1, \ldots, x^l_{N_l+\frac{1}{2}}, \ldots, x^d) \in \partial \Omega_N.
\end{align*}
As in the 1D case, we initialize $c_{ih}^{{0}}$ by projecting $c_{i}^{\rm in}$ onto $V_h$ so that
\begin{equation}\label{projMD}
\int_{K_\alpha} c_{ih}^{0}(x) v(x)dx=\int_{K_\alpha} c_{i}^{\rm in}(x)v(x)dx, \quad \forall v\in V_h, \quad i=1,\ldots, m.
\end{equation}

With these initial and boundary conditions, the scheme is now fully defined. 

\subsection{Positivity propagation}
We define the cell averages of $c_{ih}^n$ on $K_\alpha$ as 
$$
\bar{c}_{i\alpha}^{n} := \frac{1}{|K_\alpha|}\int_{K_\alpha}c_{ih}^{n}dx, 
$$
where $|K_\alpha| = \Pi_{l=1}^d \Delta x_{\alpha_l}^l$ with $\Delta x_{\alpha_l}^l = x^l_{\alpha_l+1/2}- x^l_{\alpha_l-1/2}$. 
On the interval $I^l_{\alpha_l}$, $l=1,\ldots, d$, we introduce the $M$-point Gauss-Lobatto quadrature rule using pairs $(\omega_{m_l}, x_{m_l}^{l,*})$, $1 \leq m_l \leq M$, with $M \geq \frac{k+3}{2}$. This quadrature is exact for polynomials of degree up to $k$, and the nodes satisfy
\begin{align*}
x_1^{l,*} = x_{\alpha_l-\frac{1}{2}}^l, \quad x_M^{l,*} = x_{\alpha_l+\frac{1}{2}}^l.
\end{align*}

Similar to the 1D scheme \eqref{fullyPDG}, the following results hold for the multidimensional scheme \eqref{fullyMD1}-\eqref{fullyMD3}.
\begin{theorem}\label{THMfullyPDG+}
For the fully discrete DG schemes \eqref{fullyMD1}-\eqref{fullyMD3}, the cell average $\bar c_{i\alpha }^{n+1} > 0$ provided $c_{ih}^n(x)>0$ and the CFL condition for $\forall 1\leq l \leq d$ and $\forall \ell \not =l$,

\begin{equation}\label{CFLMD}
      \Delta t \leq \frac{\omega_1}{d}
          \min  \left\{ \Delta x_{\alpha_l}^l |\widehat{\partial_{x^l} p_{ih}^n}|^{-1}\Big|_{(x_{m_1}^{1,*},\ldots, x^l_{\alpha_l \pm 1/2}, \ldots, x_{m_d}^{d,*})} \right\}, \quad 1\leq \alpha_l \leq N_l, 1\leq m_\ell \leq M,
\end{equation} 
is satisfied. Here, $\omega_1$ represents the first weight of the Gauss-Lobatto quadrature rules with $M \geq \frac{k+3}{2}$ points.
\end{theorem}
\begin{proof}
Taking $v_i = \frac{\Delta t}{|K_\alpha|}$ in \eqref{fullyMD1}, the time evolution of $\bar{c}_{i\alpha}^{n+1}$ on $K_\alpha$ is given by
\begin{align*}
\bar{c}_{i\alpha }^{n+1} = & \bar{c}_{i\alpha }^{n} + \frac{\Delta t}{|K_\alpha|}\sum_{l=1}^d \int_{K_\alpha \backslash I_{\alpha_l}^l}  \{c_{ih}^n\} \widetilde{\partial_{x^l} p_{ih}^n} \Big|_{x^l_{\alpha_l+1/2}} - \{c_{ih}^n\} \widetilde{\partial_{x^l} p_{ih}^n} \Big|_{x^l_{\alpha_l-1/2}} \,d(x\backslash x^l) \\
= & \sum_{l=1}^d \left( \frac{1}{d} \bar{c}_{ih}^{n} + \frac{\Delta t}{|K_\alpha|} \int_{K_\alpha \backslash I_{\alpha_l}^l}  \{c_{ih}^n\} \widetilde{\partial_{x^l} p_{ih}^n} \Big|_{x^l_{\alpha_l+1/2}} - \{c_{ih}^n\} \widetilde{\partial_{x^l} p_{ih}^n}  \Big|_{x^l_{\alpha_l-1/2}} \,d(x\backslash x^l) \right) \\
= & \frac{1}{d|K_\alpha|}\sum_{l=1}^d \Delta x^l_{\alpha_l} \int_{K_\alpha \backslash I_{\alpha_l}^l} \left( \frac{1}{\Delta x^l_{\alpha_l}}\int_{I_{\alpha_l}^l} {c}_{ih}^{n} dx^l + d \lambda_{\alpha_l}^l\Delta x^l_{\alpha_l}   \{c_{ih}^n\} \widetilde{\partial_{x^l} p_{ih}^n} \Big|_{x^l_{\alpha_l+1/2}} \right. \\
& \hspace{1cm} \left.- d \lambda_{\alpha_l}^l\Delta x^l_{\alpha_l}  \{c_{ih}^n\} \widetilde{\partial_{x^l} p_{ih}^n} \Big|_{x^l_{\alpha_l-1/2}} \, \right) d(x\backslash x^l),
\end{align*}
where
\begin{equation*}
\lambda^l_{\alpha_l}:= \frac{\Delta t}{(\Delta x^l_{\alpha_l})^2}.
\end{equation*}
We first apply the Gauss-Lobatto quadrature rule to approximate $\bar{c}_{i\alpha }^{n}$ on $I^{l}_{\alpha_l}$ and obtain
\begin{align*}
\bar{c}_{i\alpha }^{n+1} = & \frac{1}{d|K_\alpha|}\sum_{l=1}^d \Delta x^l_{\alpha_l} \int_{K_\alpha \backslash I_{\alpha_l}^l} \left[ \sum_{m_l=1}^M \omega_{m_l} c_{ih}^n(x^1,\ldots, x^{l,*}_{m_l}, \ldots, x^d) \right. \\
& \hspace{0cm} + \left. \frac{d \lambda_{\alpha_l}^l\Delta x^l_{\alpha_l}}{2}    \left(\widehat{\partial_{x^l} p_{ih}^n} + \widetilde{\beta_i^l} \{ c_{ih}^n \} \right) \Big|_{x_{\alpha_l+\frac{1}{2}}^{l}} c_{ih}^n(x^1, \ldots,x_{\alpha_l+\frac{1}{2}}^{l,+},\ldots, x^d) \right. \\
& \hspace{0cm} + \frac{d \lambda_{\alpha_l}^l\Delta x^l_{\alpha_l}}{2}    \left(\widehat{\partial_{x^l} p_{ih}^n} - \widetilde{\beta_i^l} \{ c_{ih}^n \} \right) \Big|_{x_{\alpha_l+\frac{1}{2}}^{l}} c_{ih}^n(x^1, \ldots,x_{\alpha_l+\frac{1}{2}}^{l,-}, \ldots, x^d) \\
& \hspace{0cm} - \frac{d \lambda_{\alpha_l}^l\Delta x^l_{\alpha_l}}{2}    \left(\widehat{\partial_{x^l} p_{ih}^n} + \widetilde{\beta_i^l} \{ c_{ih}^n \} \right) \Big|_{x_{\alpha_l-\frac{1}{2}}^{l}} c_{ih}^n(x^1, \ldots,x_{\alpha_l-\frac{1}{2}}^{l,+}, \ldots, x^d) \\
& \hspace{0cm} \left. + \frac{d \lambda_{\alpha_l}^l\Delta x^l_{\alpha_l}}{2}    \left(- \widehat{\partial_{x^l} p_{ih}^n} + \widetilde{\beta_i^l} \{ c_{ih}^n \} \right) \Big|_{x_{\alpha_l-\frac{1}{2}}^{l}} c_{ih}^n(x^1, \ldots,x_{\alpha_l-\frac{1}{2}}^{l,-}, \ldots, x^d)  \right] d(x\backslash x^l) \\
= &\frac{1}{d|K_\alpha|}\sum_{l=1}^d \Delta x^l_{\alpha_l} \int_{K_\alpha \backslash I_{\alpha_l}^l} \left[ \sum_{m_l=2}^{M-1} \omega_{m_l} c_{ih}^n(x^1,\ldots, x^{l,*}_{m_l}, \ldots, x^d) \right. \\
& \hspace{0cm} + \left. \frac{d \lambda_{\alpha_l}^l\Delta x^l_{\alpha_l}}{2}    \left(\widehat{\partial_{x^l} p_{ih}^n} + \widetilde{\beta_i^l} \{ c_{ih}^n \} \right) \Big|_{x_{\alpha_l+\frac{1}{2}}^{l}} c_{ih}^n(x^1, \ldots,x_{\alpha_l+\frac{1}{2}}^{l,+}, \ldots, x^d) \right. \\
& \hspace{0cm} + \frac{d \lambda_{\alpha_l}^l\Delta x^l_{\alpha_l}}{2}    \left(- \widehat{\partial_{x^l} p_{ih}^n} + \widetilde{\beta_i^l} \{ c_{ih}^n \} \right) \Big|_{x_{\alpha_l-\frac{1}{2}}^{l}} c_{ih}^n(x^1, \ldots,x_{\alpha_l-\frac{1}{2}}^{l,-}, \ldots, x^d) \\
& \hspace{0cm} + \left( \omega_1- \frac{d \lambda_{\alpha_l}^l\Delta x^l_{\alpha_l}}{2}    \left(\widehat{\partial_{x^l} p_{ih}^n} + \widetilde{\beta_i^l} \{ c_{ih}^n \} \right) \Big|_{x_{\alpha_l-\frac{1}{2}}^{l}} \right) c_{ih}^n(x^1, \ldots,x_{\alpha_l-\frac{1}{2}}^{l,+}, \ldots, x^d) \\
& \hspace{0cm} \left. + \left( \omega_M - \frac{d \lambda_{\alpha_l}^l\Delta x^l_{\alpha_l}}{2}    \left(-\widehat{\partial_{x^l} p_{ih}^n} + \widetilde{\beta_i^l} \{ c_{ih}^n \} \right) \Big|_{x_{\alpha_l+\frac{1}{2}}^{l}} \right)c_{ih}^n(x^1, \ldots,x_{\alpha_l+\frac{1}{2}}^{l,-}, \ldots, x^d)  \right] d(x\backslash x^l).
\end{align*}
Based on the choice of $\widetilde{\beta_i^l}$ in \eqref{defbetaM}, it follows
\begin{align*}
\left(\widehat{\partial_{x^l} p_{ih}^n} + \widetilde{\beta_i^l} \{ c_{ih}^n \} \right) \Big|_{(x^1, \ldots,x_{\alpha_l+\frac{1}{2}}^{l}, \ldots, x^d)} \geq 0, \quad \left(-\widehat{\partial_{x^l} p_{ih}^n} + \widetilde{\beta_i^l} \{ c_{ih}^n \} \right) \Big|_{(x^1, \ldots,x_{\alpha_l-\frac{1}{2}}^{l}, \ldots, x^d)} \geq 0.
\end{align*}
Then applying the Gauss-Lobatto quadrature rule in all directions except $x^l$ gives
\begin{align*}
\bar{c}_{i\alpha }^{n+1} = &\frac{1}{d|K_\alpha|}\sum_{l=1}^d \Delta x^l_{\alpha_l}  \sum_{m_1=1}^M \omega_{m_1} \cdots \sum_{m_{l-1}=1}^M \omega_{m_{l-1}} \sum_{m_{l+1}=1}^M \omega_{m_{l+1}}\cdots \sum_{m_d=1}^M \omega_{m_d} \\
& \hspace{-1cm} \cdot \left[ \sum_{m_l=2}^{M-1} \omega_{m_l} c_{ih}^n(x^{1,*}_{m_1},\ldots, x^{l,*}_{m_l}, \ldots, x^{d,*}_{m_d}) \right. \\
& \hspace{-1cm} + \left. \frac{d \lambda_{\alpha_l}^l\Delta x^l_{\alpha_l}}{2}    \left(\widehat{\partial_{x^l} p_{ih}^n} + \widetilde{\beta_i^l} \{ c_{ih}^n \} \right) \Big|_{(x^{1,*}_{m_1}, \ldots,x_{\alpha_l+\frac{1}{2}}^{l}, \ldots, x^{d,*}_{m_d})} c_{ih}^n(x^{1,*}_{m_1}, \ldots,x_{\alpha_l+\frac{1}{2}}^{l,+}, \ldots, x^{d,*}_{m_d}) \right. \\
& \hspace{-1cm} + \frac{d \lambda_{\alpha_l}^l\Delta x^l_{\alpha_l}}{2}    \left(- \widehat{\partial_{x^l} p_{ih}^n} + \widetilde{\beta_i^l} \{ c_{ih}^n \} \right) \Big|_{(x^{1,*}_{m_1}, \ldots,x_{\alpha_l-\frac{1}{2}}^{l}, \ldots, x^{d,*}_{m_d})} c_{ih}^n(x^{1,*}_{m_1}, \ldots,x_{\alpha_l-\frac{1}{2}}^{l,-}, \ldots, x^{d,*}_{m_d}) \\
& \hspace{-1cm} + \left( \omega_1- \frac{d \lambda_{\alpha_l}^l\Delta x^l_{\alpha_l}}{2}    \left(\widehat{\partial_{x^l} p_{ih}^n} + \widetilde{\beta_i^l} \{ c_{ih}^n \} \right) \Big|_{(x^{1,*}_{m_1}, \ldots,x_{\alpha_l-\frac{1}{2}}^{l}, \ldots, x^{d,*}_{m_d})} \right) c_{ih}^n(x^{1,*}_{m_1}, \ldots,x_{\alpha_l-\frac{1}{2}}^{l,+}, \ldots, x^{d,*}_{m_d}) \\
& \hspace{-1cm} \left. + \left( \omega_M - \frac{d \lambda_{\alpha_l}^l\Delta x^l_{\alpha_l}}{2}    \left(-\widehat{\partial_{x^l} p_{ih}^n} + \widetilde{\beta_i^l} \{ c_{ih}^n \} \right) \Big|_{(x^{1,*}_{m_1}, \ldots,x_{\alpha_l+\frac{1}{2}}^{l}, \ldots, x^{d,*}_{m_d})} \right)c_{ih}^n(x^{1,*}_{m_1}, \ldots,x_{\alpha_l+\frac{1}{2}}^{l,-}, \ldots, x^{d,*}_{m_d}))  \right] \\
& \hspace{-1cm} >0,
\end{align*}
provided
\begin{align*}
\lambda_{\alpha_l}^l \leq \frac{1}{d} \frac{\omega_1}{\Delta x^l_{\alpha_l}} \min\left\{  |\widehat{\partial_{x^l} p_{ih}^n}|^{-1} \Big|_{(x_{m_1}^{1,*},\ldots, x^l_{\alpha_l \pm 1/2}, \ldots, x_{m_d}^{d,*})} \right\}, \quad  1\leq m_\ell \leq M, \quad \forall \ell \not = l.
 \end{align*}
Again, we used $\omega_1 = \omega_N$ due to the symmetry of Gauss-Lobatto quadrature rules.

\end{proof}

\begin{remark}\label{old2Dscheme}
If $\widetilde{\beta_i^l}\equiv 0$ in \eqref{betaMD}, then the DG scheme \eqref{fullyMD1}-\eqref{fullyMD3} reduces to the high-dimensional form of the fully discrete DG scheme described in \eqref{dgEuler}. 
\end{remark}

\section{Numerical Examples} \label{sec:num}
In this section, we present a set of selected examples to validate our positivity-preserving DDG scheme. For accuracy, we assess the order of accuracy by performing numerical convergence tests using discrete $l_1$ errors.
\subsection{1D accuracy tests}
We begin to study the impact of the positivity preserving flux $\widetilde{\partial_x p_{ih}}$ on the scheme accuracy.  We revisit the example from \S 5.1 of \cite{LW17}, where the DDG scheme \eqref{dgEuler} was applied. Consider the domain $\Omega=[0,1]$ and the PNP problem with source terms defined by the following system of equations:  
\begin{align*}
&  \partial_t c_1  = \partial_x (\partial_x c_1+  q_1 c_1 \partial_x \psi) +f_1, \\
&  \partial_t c_2  = \partial_x (\partial_x c_2+   q_2 c_2 \partial_x \psi) +f_2, \\
&- \partial_x^2 \psi =  q_1c_1+ q_2c_2,\\
&\partial_x \psi(t,0) =0,  \quad \partial_x \psi(t, 1)=-e^{-t}/60, \\
&{\partial_x c_i} +q_ic_i {\partial_x \psi}=0,
 \quad x=0, 1,
\end{align*}
where the source terms are given by 
\begin{align*}
 f_1 &=\frac{(50x^9-198x^8+292x^7-189x^6+45x^5)}{30e^{2t}}+\frac{ (-x^4+2x^3-13x^2+12x-2)}{e^{t}},   \\
 f_2&=\frac{(x - 1)(110x^9 - 430x^8 + 623x^7 - 393x^6+90x^5)}{60e^{2t}} + \frac{(x-1)(x^4 - 2x^3 + 21x^2 - 16x + 2)}{e^{t}}.
\end{align*}
This system, with $q_1=1$ and $q_2=-1$, admits exact solutions:
\begin{align*}
  c_1 & =x^2(1-x)^2e^{-t}, \\
 c_2  &= x^2(1-x)^3e^{-t}, \\
\psi & =  -(10x^7-28x^6+21x^5)e^{-t}/420.
\end{align*}
Note that we impose  $\psi(t,0)=0$ to select a specific solution, as  $\psi$ is unique up to an additive constant. 

We solve this problem using the positivity-preserving DG scheme \eqref{fullyPDG}.
Table \ref{tab:ex1} presents both $l_1$ errors and orders of convergence when using $P^{k}$ elements at $T=0.05$, incorporating the modified local flux $\widetilde{\partial_x p_{ih}}$. Our results indicate that the method preserves an order of convergence of $k+1$, consistent with the findings in \cite{LW17}. This suggests that the modified flux $\widetilde{\partial_x p_{ih}}$ maintains optimal $k+1$ order of accuracy. We adhere to the CFL conditions outlined in Theorem \ref{thmEuler} and Remark \ref{rem:CFL}. Figure \ref{fig:ex1} provides a visual comparison between numerical and exact solutions at $t=0.05$.

\begin{table}[!htb]
\caption{Errors for Example 1 at $T=0.05$}
\begin{tabular}{ |c|l|c|c| c|c| c|c| }
\hline
$(k,\beta_0, \beta_1)$& h & $c_1$ error & order & $c_2$ error & order & $\psi$ error & order \\ \hline
\multirow{4}{*}{$(1,2,-)$}
&  0.2&  0.00074325&  -&  0.001946&  -&  0.00044667&  -\\ 
&  0.1&  0.00015581&  2.0164&  0.00022054&  2.6233&  4.2569e-05&  3.3250\\ 
&  0.05&  4.3138e-05&  1.8976&  3.8745e-05&  2.3642&  4.1411e-06&  3.2918\\ 
&  0.025&  1.1223e-05&  1.9425&  8.3198e-06&  2.2194&  4.4382e-07&  3.2220\\ \hline
 \multirow{4}{*}{$(2,4,{1}/{20})$}
&  0.2&  0.0035405&  -&  0.0013336&  -&  0.00026474&  -\\ 
&  0.1&  0.00075498&  2.5660&  0.00022787&  2.6867&  4.1164e-05&  2.7921\\ 
&  0.05&  0.00011782&  2.7343&  3.4364e-05&  2.7555&  5.7354e-06&  2.8456\\ 
&  0.025&  1.7049e-05&  2.7889&  4.9971e-06&  2.7817&  7.9669e-07&  2.8478\\  \hline
\end{tabular}
\label{tab:ex1}
\end{table}

 \begin{figure}[!htb]
 \caption{Numerical solution versus exact solution at $T=0.05$}
 \centering
 \begin{tabular}{cc}
 \includegraphics[width=\textwidth]{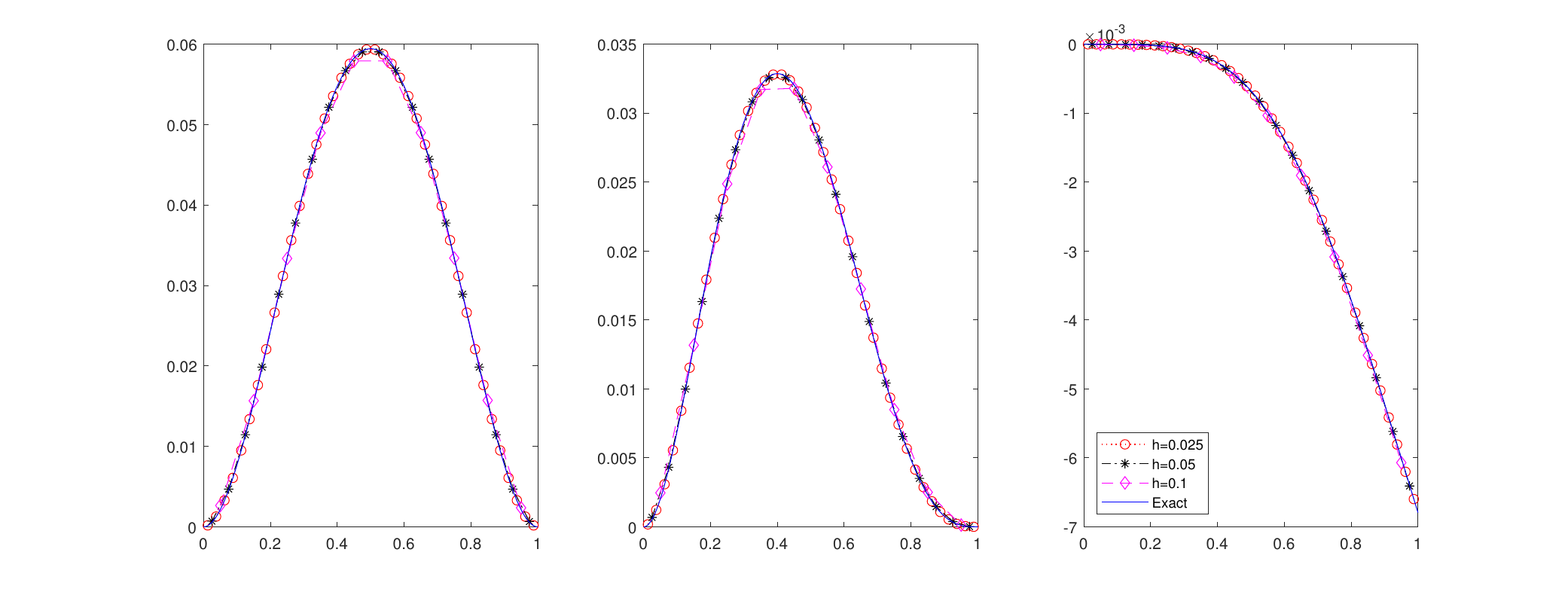}\\
 \end{tabular}
  \label{fig:ex1} \end{figure}

\subsection{1D tests on solution properties}
\label{1Dtestacc}
We test solution properties  
such as solution positivity, mass conservation and free energy dissipation. We apply the proposed scheme to the one-dimensional PNP system  \eqref{PNP} with $m=2$ in $\Omega=[0,1]$, defined by the following equations:  
\begin{align*}
&\partial_t c_1 = \partial_x (\partial_x c_1+ q_1 c_1 \partial_x \psi), \\
&\partial_t c_2  = \partial_x (\partial_x c_2 + q_2 c_2 \partial_x \psi), \\
& - \partial_x^2 \psi =  q_1c_1 + q_2c_2,\\
& c_1(0,x) = c_1^{\rm in}(x), \quad 
 c_2(0,x) = c_2^{\rm in}(x), \\
&\psi(t,0) =0,  \quad \partial_x \psi(t,0) =0,  \quad \partial_x \psi(t, 1)=0, \\
& {\partial_x c_i} +q_ic_i {\partial_x \psi}=0,
 \quad x=0, \; 1.
\end{align*}
In this system, the ion charges are $q_1 = 1$ and $q_2 = -1$, and the initial concentrations  are given by 
\begin{equation*}
\begin{aligned}
c_1^{\rm in}(x)= & \left\{ {\begin{array}{*{20}{l}}
{ 0.1, \quad \text{if } x \in (0.4, 0.6),} \\
{ 0.288, \quad \text{if } x \in [0.2, 0.4] \cup [0.6, 0.8], }\\
{ 5x^2(1-x)^2, \quad \text{otherwise}, }
\end{array}} \right.\\
c_2^{\rm in}(x)= & \frac{\pi}{10} \left| \sin( 2\pi x^2 ) \right|.
\end{aligned}
\end{equation*}
For the DG schemes applied to this example,  we use $P^1$ polynomials with numerical flux parameters $\beta_0 =4$. The initial time step is set to $\Delta t_0 = 3.5\times 10^{-5}$, and the mesh size is $h=0.025$.

\noindent \textbf{Test case 1}. We solve this problem using the DG scheme \eqref{fullyPDG} with a fixed time step $\Delta t = \Delta t_0$ and employ the modified flux $\widetilde{\partial_x p_{ih}}$ in \eqref{Pflux} at each step. 
\Cref{ex2cgall} shows the evolution of the smallest cell averages of $c_1$ and $c_2$, demonstrating that the DG scheme \eqref{fullyPDG} conserves the positivity of the solutions at
quadrature points. 
\begin{figure}
\centering
\subfigure[Smallest cell average of $c_1$]{\includegraphics[width=0.49\textwidth]{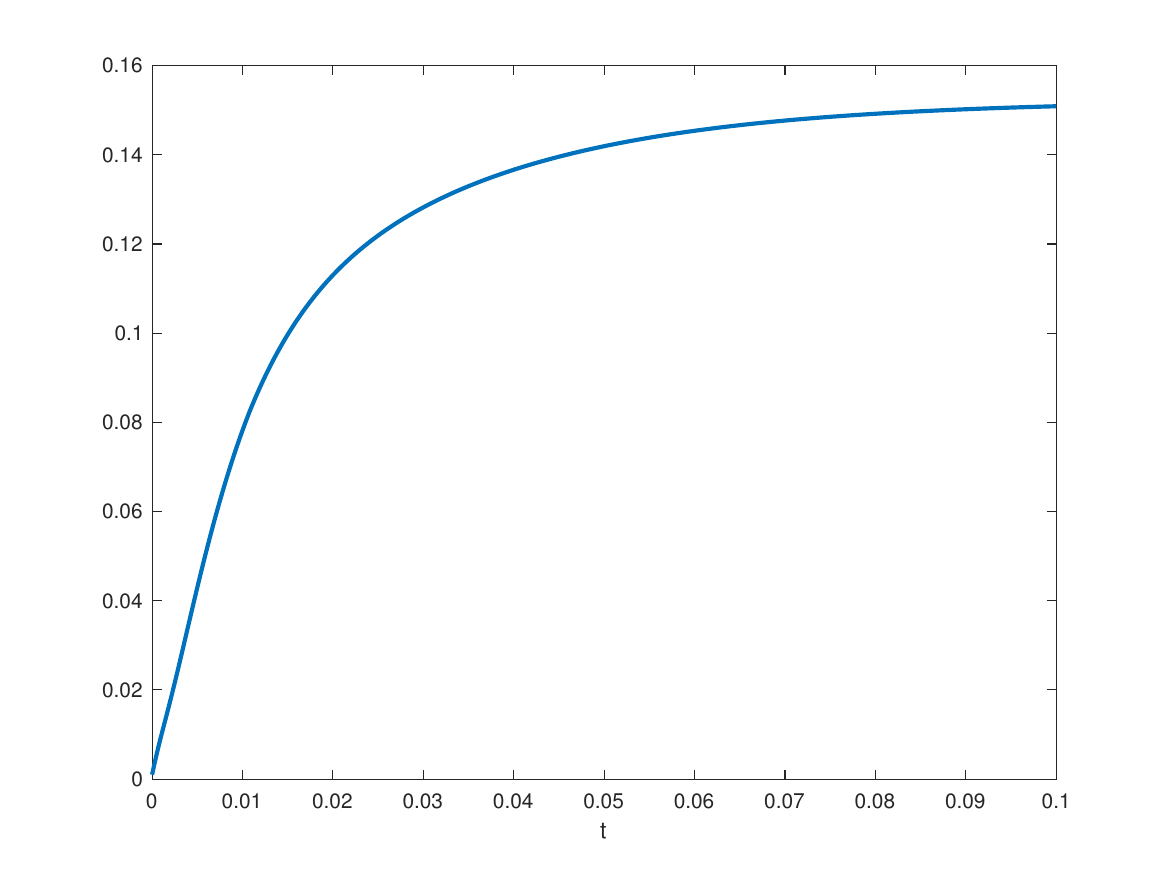}}
\subfigure[Smallest cell average of $c_2$]{\includegraphics[width=0.49\textwidth]{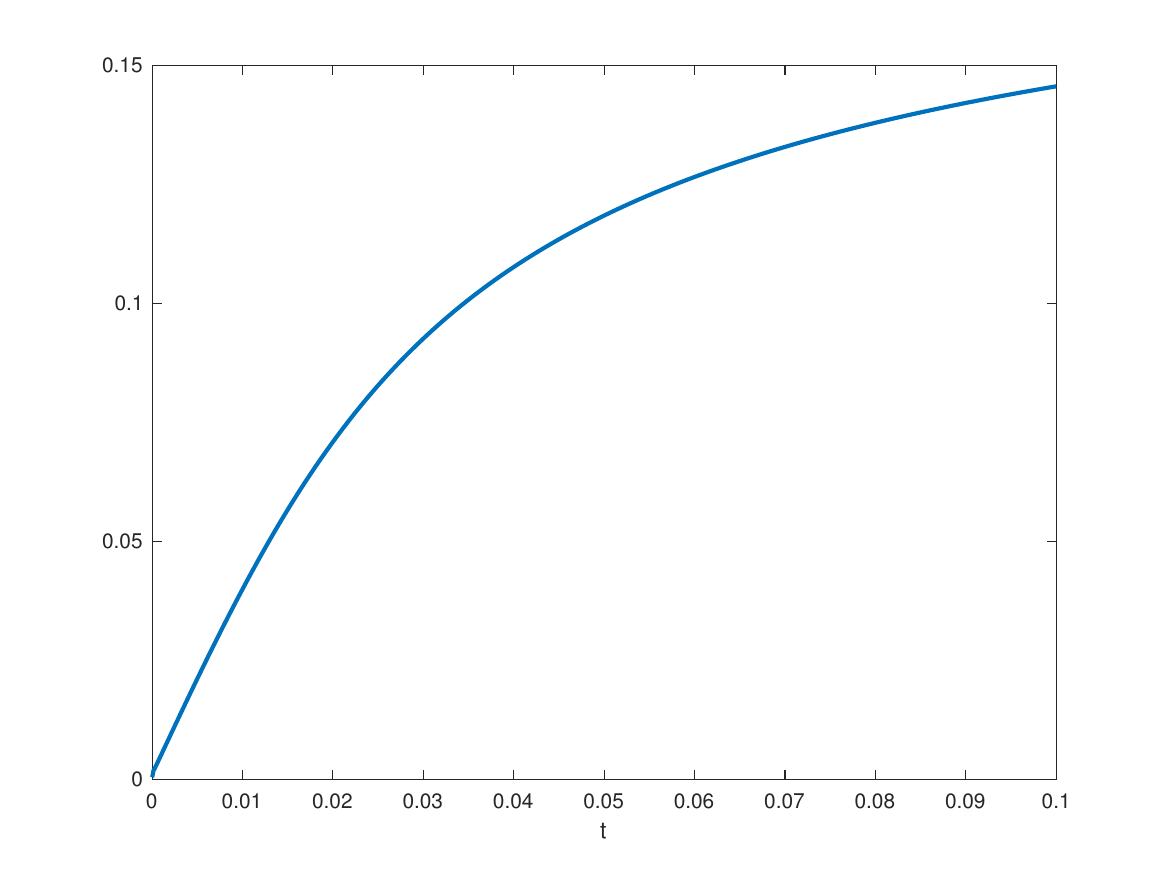}}
\caption{ Smallest cell averages of $c_1, c_2$.
  } \label{ex2cgall}
 \end{figure}

\Cref{ex2emall} displays the evolution of the total masses of $c_1$ and $c_2$, which are conserved at approximately $0.15439$ for $c_1$ and $0.17031$ for $c_2$. Additionally, Figure \ref{ex2emall} confirms the dissipation of free energy. \Cref{fig:ex2solall} shows the evolution of the solutions, indicating that  $c_1$ and $c_2$ approach equilibrium at $T=0.1$.

 \begin{figure}[!htb]
\caption{Conservation of mass and decay of free energy}
\centering
\begin{tabular}{cc}
\includegraphics[width=\textwidth]{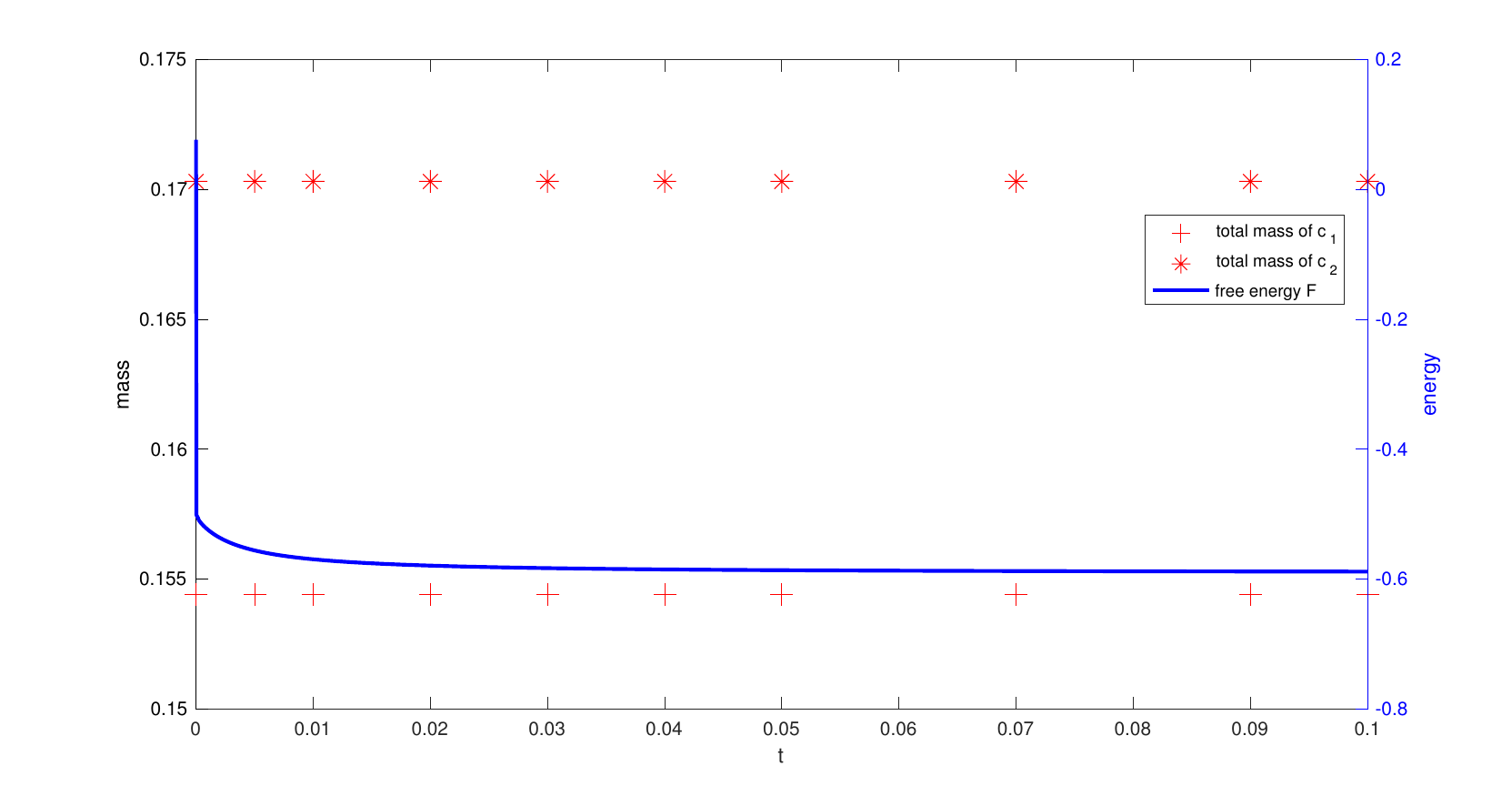} 
\end{tabular}
\label{ex2emall}
\end{figure}

 \begin{figure}[!htb]
 \caption{Numerical solution evolution}
 \centering
 \begin{tabular}{cc}
 \includegraphics[width=\textwidth]{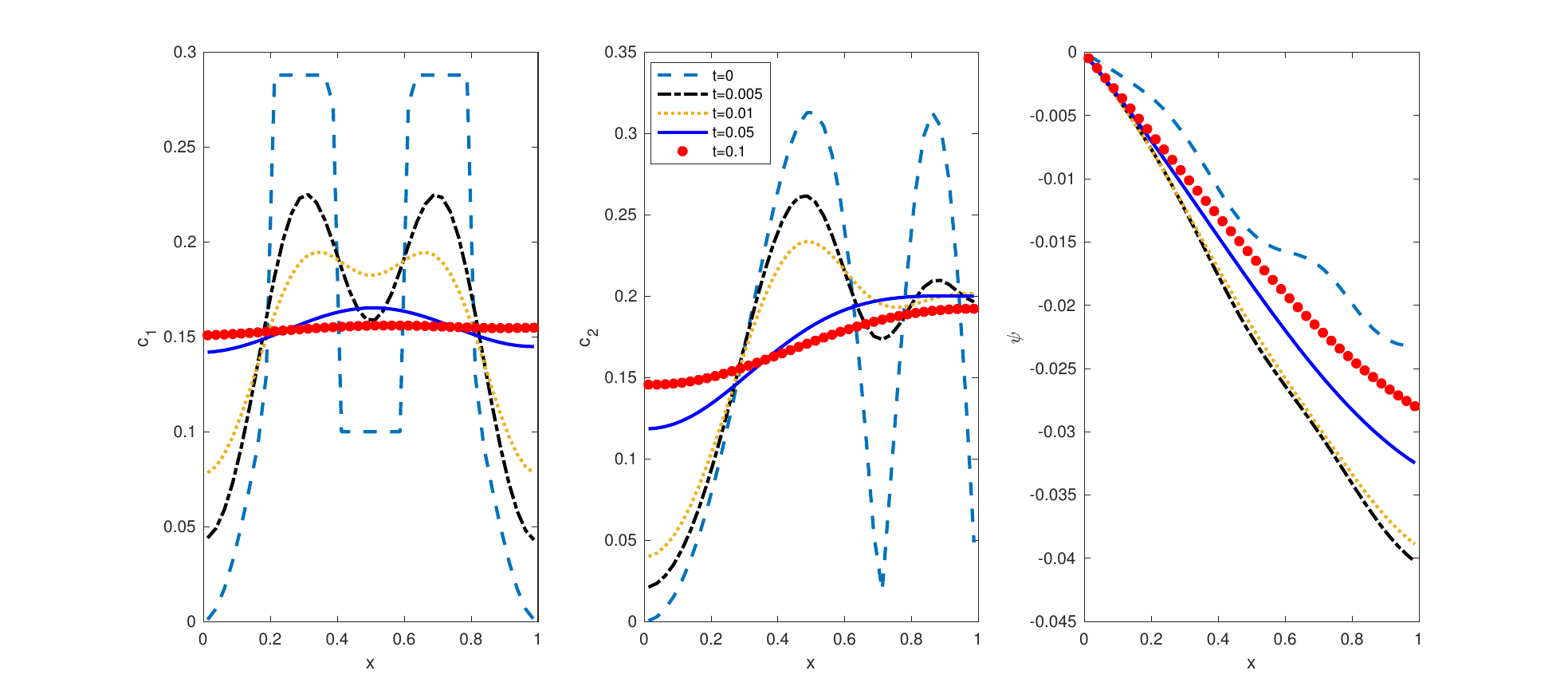}\\
 \end{tabular}
  \label{fig:ex2solall} \end{figure}

\noindent \textbf{Test case 2}. Next, we solve this problem using the DG scheme \eqref{dgEuler} from our previous work \cite{LW17} with the same fixed time step $\Delta t = \Delta t_0$ as used in Test case 1. However, the numerical results indicate that at $t=0.000245$, the smallest cell average of the concentration $c_2$ on the quadrature points is $-0.029306$, indicating a failure of the DG scheme \eqref{dgEuler}. 

To address this issue, we modify the algorithm for the DG scheme \eqref{dgEuler}. Specifically, if the smallest cell averages of $c_1$ and $c_2$ drop below $0$ (or small threshold $\varepsilon$), we
switch to using the DG scheme \eqref{fullyPDG} at that particular time step. Additionally, we adjust the time step according to the CFL condition specified in \Cref{CFL1D} or \Cref{rem:CFL} with $\gamma=1$.

Numerically, \Cref{ex2cgpart} shows that the modified algorithm effectively preserves the positivity of $c_1$ and $c_2$. \Cref{ex2empart} illustrates the decay of the free energy and confirms the conservation of total masses $0.15439$ for $c_1$ and $0.17031$ for $c_2$. Furthermore,   \Cref{fig:ex2solpart} shows a similar evolution of the solutions as observed in Test case 1.

\begin{figure}
\centering
\subfigure[Smallest cell average of $c_1$]{\includegraphics[width=0.49\textwidth]{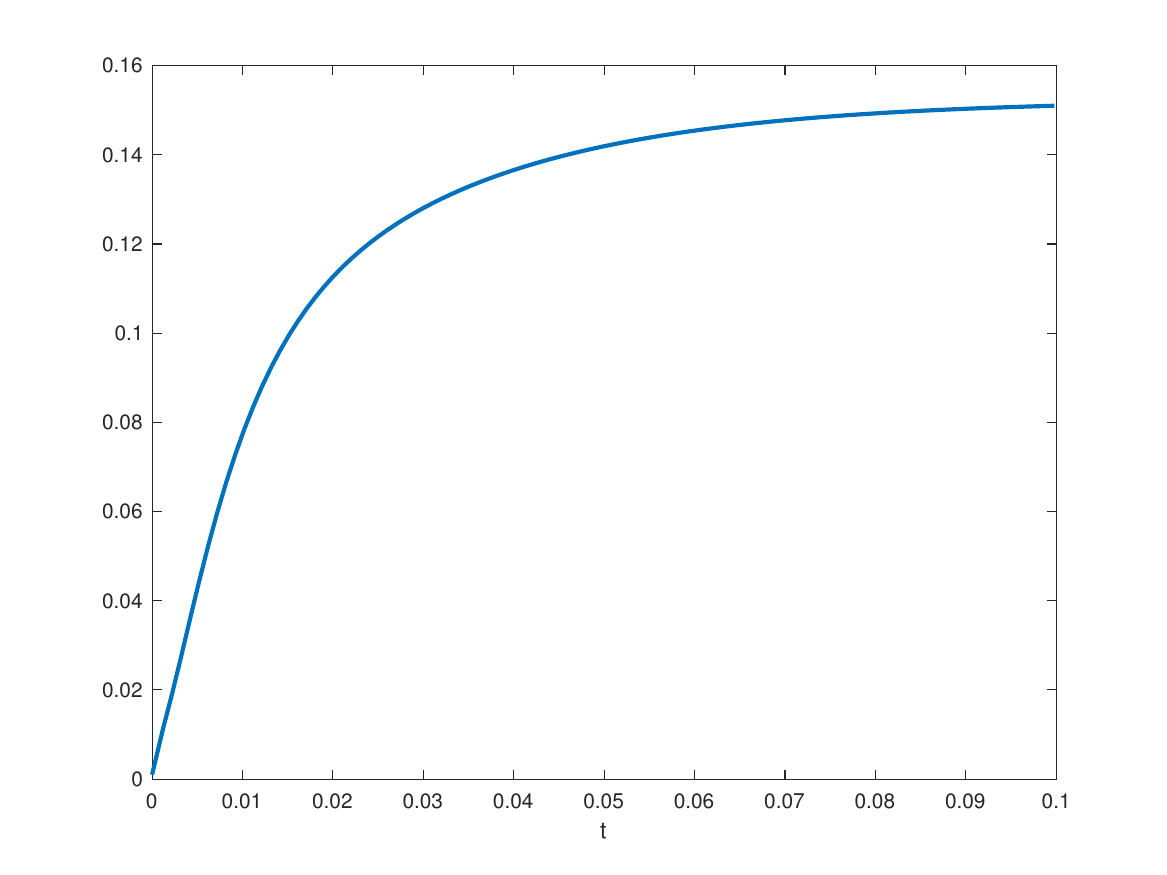}}
\subfigure[Smallest cell average of $c_2$]{\includegraphics[width=0.49\textwidth]{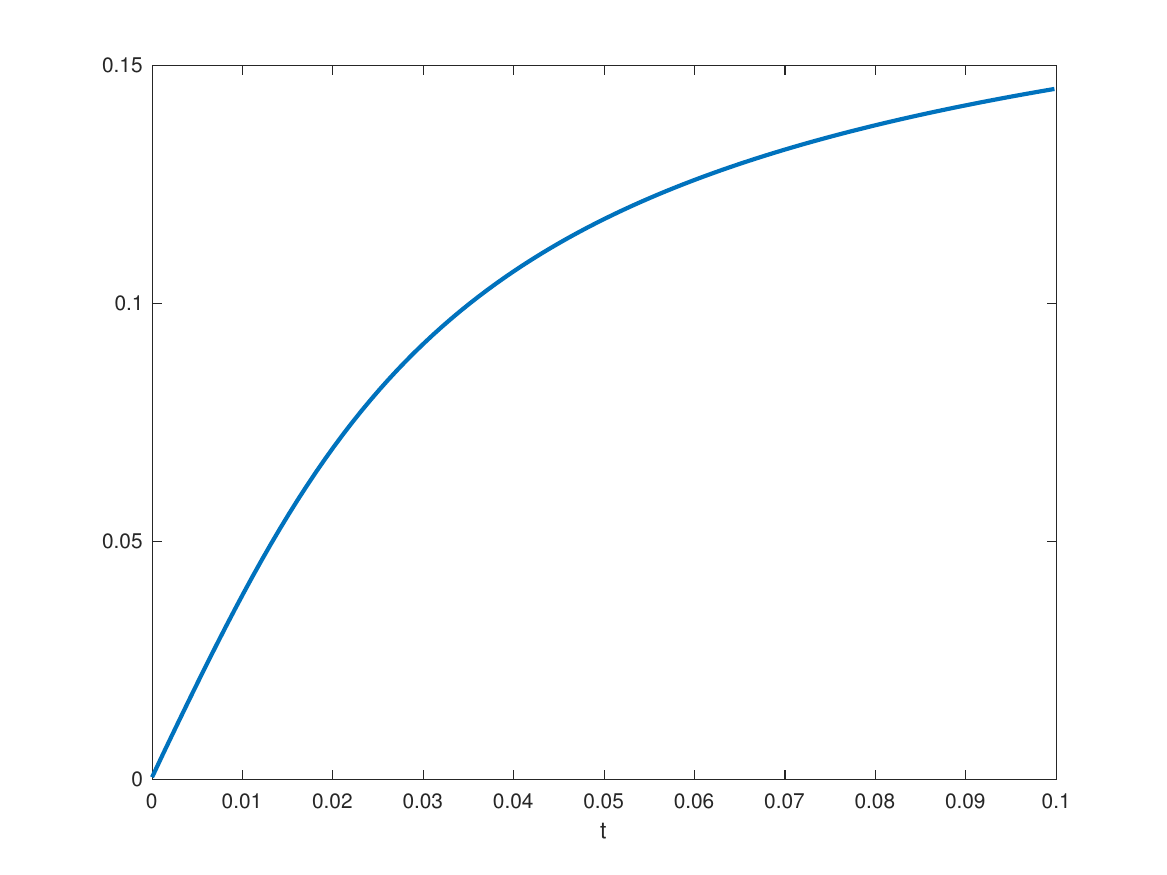}}
\caption{ Smallest cell averages of $c_1, c_2$.
  } \label{ex2cgpart}
 \end{figure}

 \begin{figure}[!htb]
\caption{Conservation of mass and decay of free energy}
\centering
\begin{tabular}{cc}
\includegraphics[width=\textwidth]{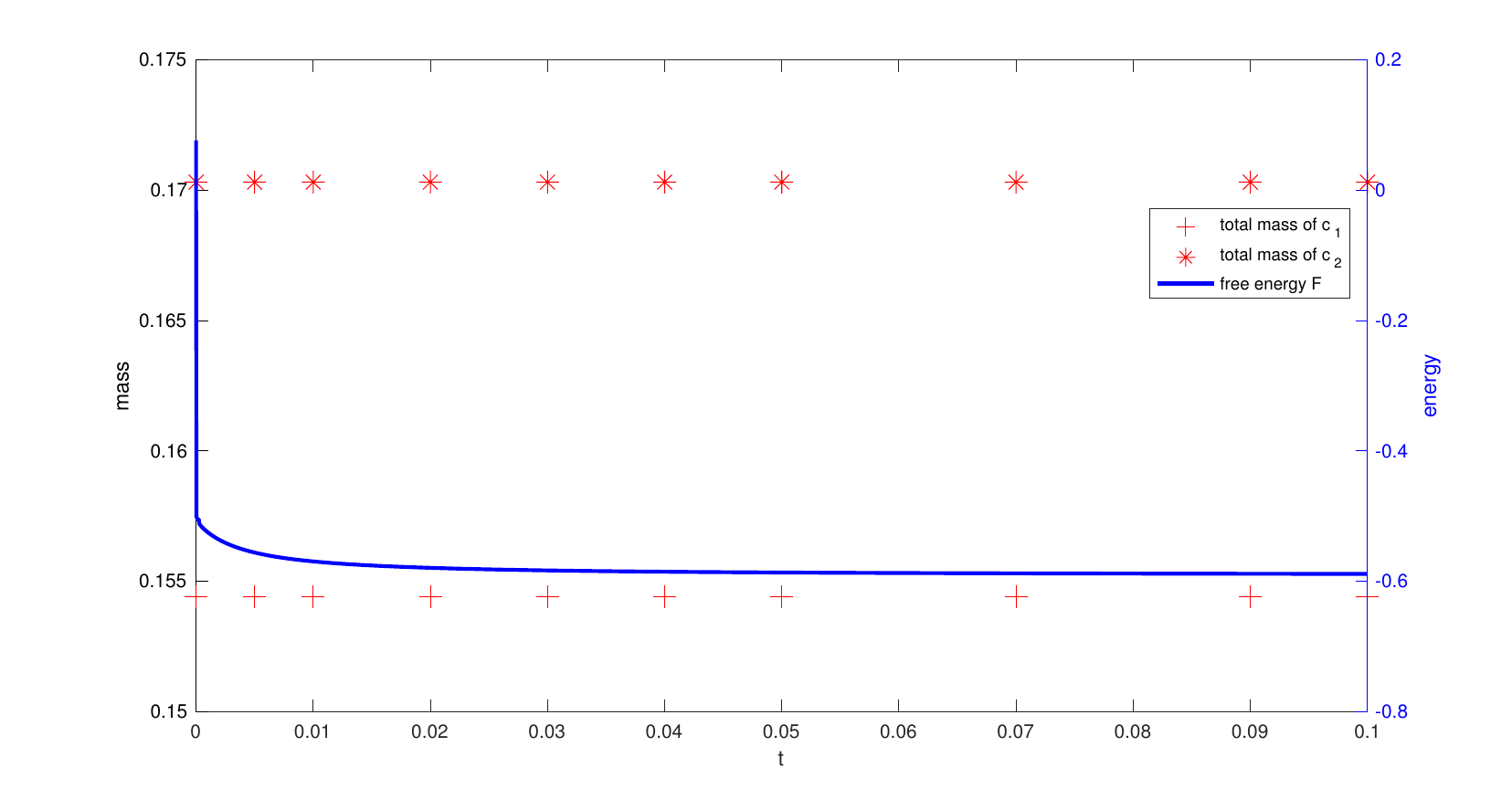} 
\end{tabular}
\label{ex2empart}
\end{figure}

 \begin{figure}[!htb]
 \caption{Numerical solution evolution}
 \centering
 \begin{tabular}{cc}
 \includegraphics[width=\textwidth]{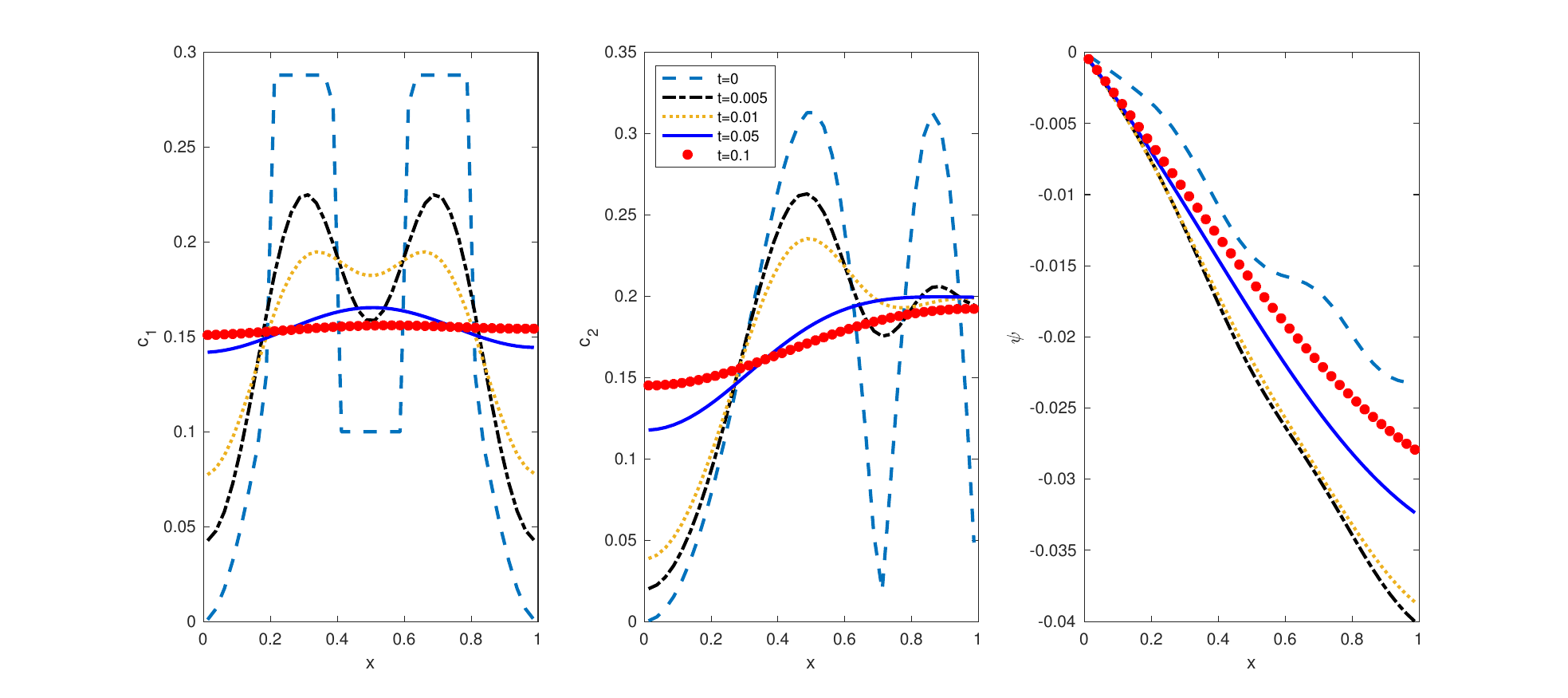}\\
 \end{tabular}
  \label{fig:ex2solpart} \end{figure}

From \textbf{Test case 1} and \textbf{Test case 2}, we see that our positivity preserving flux $\widetilde{\partial_x p_{ih}}$ will prevent negative cell average when applied at each time step following \textbf{Algorithm \ref{alg:HDDG}}, or only when needed  (i.e., when \eqref{dgEuler} from  \cite{LW17} fails), while maintaining order of accuracy.

\subsection{2D accuracy tests}\label{2Dtestacc}
We evaluate the spatial accuracy of our scheme in a 2D setting using the PNP problem \eqref{PNP} defined on $\Omega=[0,\pi]^2$,  with prescribed source terms:  
\begin{align*}
&\partial_t c_1= \nabla\cdot(\nabla c_1+c_1\nabla\psi)+f_1,\\
&\partial_t c_2= \nabla\cdot(\nabla c_2-c_2\nabla\psi)+f_2,\\
&-\Delta \psi =  c_1-c_2 + f_3.
\end{align*}
Here,the functions $f_i(t,x,y)$ are determined by the exact solutions $c_1(t,x,y)$, $c_2(t,x,y)$, and $\psi(t,x,y)$, which are specified for each test case.
The initial conditions in (\ref{PNP}c) are derived by evaluating the exact solution at $t=0$, while the boundary conditions in (\ref{PNP}c) are set to zero flux. 
We denote the boundary sets as $\partial \Omega_D=\{(x,y)\in \bar{\Omega}: x=0, x=\pi \}$ and $\partial \Omega_N=\partial \Omega \backslash \partial \Omega_D$. The boundary data in (\ref{PNP}d) is derived from evaluating the exact solution $\psi(t,x,y)$ on $\partial \Omega_D$,  and its normal derivative $\frac{\partial \psi}{\partial  \textbf{n}}$ on $\partial \Omega_N$. 
The problem is tackled using the fully discrete DG scheme \eqref{dgEuler} in a 2D setting (for more details, refer to \Cref{old2Dscheme}), along with the DG scheme \eqref{fullyMD1}-\eqref{fullyMD3}, that has been adjusted to incorporate the specific source terms. 

{\bf \noindent Test case 1.}
We consider the exact solutions: 
\begin{align*}
& c_1(t,x,y) = \alpha_{1} \left( e^{-\alpha t} \cos(x)\cos(y) + 1\right), \\
& c_2(t,x,y)= \alpha_{2} \left( e^{-\alpha t} \cos(x)\cos(y) + 1\right), \\
& \psi(t,x,y) = \alpha_{3} e^{-\alpha t} \cos(x)\cos(y),
\end{align*}
where the parameters  are set as $\alpha=\alpha_{1}=2\alpha_{2}=\alpha_{3}=10^{-2}$. It is observed that $c_1 \geq 0$ and $c_2\geq 0$ for $t\geq 0$, with  $\min_{(x,y)}{c_1} = 0$ and $\min_{(x,y)}{c_2} = 0$ achievable only at $t=0$.

The numerical flux parameters $\beta_0, \beta_1$ in the DDG schemes follow the guidelines from  \cite{YHL14, YHL18},  and the specific values  of $\beta_0, \beta_1$ can be found in \Cref{tab:ex2Ds1LT} and \Cref{tab:ex2Ds1LT2}. A time step of $\Delta t  = 10^{-(k+1)} h^2$ is chosen for both DG schemes. 

\Cref{tab:ex2Ds1LT} presents the $l_1$ errors and orders of convergence at $t=0.1$ for the DG scheme \eqref{dgEuler} in the 2D setting with $N\times N$ meshes.  Additionally, results for the DG scheme \eqref{fullyMD1}-\eqref{fullyMD3}, using the numerical fluxes \eqref{betaMD} at each time step, are reported in \Cref{tab:ex2Ds1LT2}.
These results demonstrate that both schemes achieve $(k+1)$-th order of accuracy in space, confirming that the modified  numerical fluxes do not compromise the accuracy of the numerical solutions. 

\begin{table}[!htb]
\caption{Errors and orders of convergence for test case 1 in \Cref{2Dtestacc} for the DG scheme \eqref{dgEuler} in the 2D setting with meshes $N\times N$ at $T=0.1$}
\begin{tabular}{ |c|l|c|c| c|c| c|c| }
\hline
$(k,\beta_0, \beta_1)$& $N$ & $c_1$ error & order & $c_2$ error & order & $\psi$ error & order \\ \hline
\multirow{4}{*}{$(1,3,-)$}
 & $10$ & 4.35298e-04 &  --  & 2.17789e-04 &  --  & 6.23731e-04 &  --   \\
 & $20$ & 1.09438e-04 & 1.99 & 5.47851e-05 & 1.99 & 1.67117e-04 & 1.90  \\
 & $30$ & 4.86339e-05 & 2.00 & 2.43496e-05 & 2.00 & 7.55204e-05 & 1.96  \\
 & $40$ & 2.73409e-05 & 2.00 & 1.36889e-05 & 2.00 & 4.27726e-05 & 1.98  \\
 \hline
 \multirow{4}{*}{$(2,9,\frac{1}{12})$}
 & $10$ & 2.62719e-05 &  --  & 1.31419e-05 &  --  & 2.05689e-05 &  --   \\
 & $20$ & 3.17503e-06 & 3.05 & 1.58789e-06 & 3.05 & 2.38879e-06 & 3.11  \\
 & $30$ & 9.09636e-07 & 3.08 & 4.54895e-07 & 3.08 & 6.92827e-07 & 3.05  \\
 & $40$ & 3.75304e-07 & 3.08 & 1.87676e-07 & 3.08 & 2.89435e-07 & 3.03  \\
 \hline
 \multirow{4}{*}{$(3,19,\frac{1}{12})$}

 & 10 & 5.34967e-06 &  --  & 2.67516e-06 &  --  & 7.94054e-07 &  --   \\
 & 20 & 3.60228e-07 & 3.89 & 1.80121e-07 & 3.89 & 4.82770e-08 & 4.04  \\
 & 30 & 6.93756e-08 & 4.06 & 3.46887e-08 & 4.06 & 9.46406e-09 & 4.02  \\
 & 40 & 2.07955e-08 & 4.19 & 1.03979e-08 & 4.19 & 2.98476e-09 & 4.01  \\ 
 \hline
\end{tabular}
\label{tab:ex2Ds1LT}
\end{table}

\begin{table}[!htb]
\caption{Errors and orders of convergence for test case 1 in \Cref{2Dtestacc} for the DG scheme \eqref{fullyMD1}-\eqref{fullyMD3} with meshes $N\times N$ at $T=0.1$}
\begin{tabular}{ |c|l|c|c| c|c| c|c| }
\hline
$(k,\beta_0, \beta_1)$& $N$ & $c_1$ error & order & $c_2$ error & order & $\psi$ error & order \\ \hline
\multirow{4}{*}{$(1,3,-)$}
 & $10$ & 4.36263e-04 &  --  & 2.18339e-04 &  --  & 6.23863e-04 &  --   \\
 & $20$ & 1.09719e-04 & 1.99 & 5.49267e-05 & 1.99 & 1.67132e-04 & 1.90  \\
 & $30$ & 4.87287e-05 & 2.00 & 2.43965e-05 & 2.00 & 7.55255e-05 & 1.96  \\
 & $40$ & 2.73888e-05 & 2.00 & 1.37128e-05 & 2.00 & 4.27750e-05 & 1.98  \\

\hline
 \multirow{4}{*}{$(2,9,{1}/{12})$}
 & $10$ & 2.61752e-05 &  --  & 1.30945e-05 &  --  & 2.05701e-05 &  --   \\
 & $20$ & 3.16680e-06 & 3.05 & 1.58381e-06 & 3.05 & 2.38883e-06 & 3.11  \\
 & $30$ & 9.07046e-07 & 3.08 & 4.53606e-07 & 3.08 & 6.92832e-07 & 3.05  \\
 & $40$ & 3.74168e-07 & 3.08 & 1.87111e-07 & 3.08 & 2.89436e-07 & 3.03  \\
 \hline
 \multirow{4}{*}{$(3,19,{1}/{12})$}

 & 10 & 5.29795e-06 &  --  & 2.64941e-06 &  --  & 7.94046e-07 &  --   \\
 & 20 & 3.57902e-07 & 3.89 & 1.78959e-07 & 3.89 & 4.82773e-08 & 4.04  \\
 & 30 & 6.90252e-08 & 4.06 & 3.45136e-08 & 4.06 & 9.46415e-09 & 4.02  \\
 & 40 & 2.07029e-08 & 4.19 & 1.03517e-08 & 4.19 & 2.98479e-09 & 4.01  \\
 \hline
\end{tabular}
\label{tab:ex2Ds1LT2}
\end{table}

{\bf \noindent Test case 2.}
We modify the exact solution from Test Case 1 as
follows: 
\begin{align*}
& c_1(t,x,y) = \alpha_{1} e^{-\alpha t} \left( \cos(x)\cos(y) + 1\right), \\
& c_2(t,x,y)= \alpha_{2}  e^{-\alpha t} \left( \cos(x)\cos(y) + 1\right), \\
& \psi(t,x,y) = \alpha_{3} e^{-\alpha t} \cos(x)\cos(y),
\end{align*}
where the parameters are set as $\alpha=\alpha_{1}=2\alpha_{2}=\alpha_{3}=10^{-2}$.
It can be verified that the densities  $c_1 \geq 0$ and $c_2\geq 0$ for $t\geq 0$, and $\min_{(x,y)}{c_1} = 0$ and $\min_{(x,y)}{c_2} = 0$  are achieved for any $t \geq 0$.

For this test case, we consider two algorithms based on $P^2$ polynomials. Following the guidelines in \cite{LWYY22} for the third-order positivity-preserving DDG method for PNP equations,  we select the numerical flux parameters $\beta_0=16$ and $\beta_1=\frac{1}{6}$.
The initial time step is set according to the DDG method in \cite{LWYY22} to ensure positive  average values of densities $c_1$ and $c_2$ in each computational cell,  specifically $\Delta t  = 5\times 10^{-3} h^2$ for both DG schemes in this paper. Note that this time step may not guarantee that the DDG scheme in \cite{LWYY22} achieves the optimal orders of convergence. 

First, we evaluate the DG scheme \eqref{dgEuler} in the 2D setting. The evolution of the minimum cell average values of $c_1$ and $c_2$ on quadrature points with various mesh sizes is presented in \Cref{tab:c12cellavg}. It is observed that the scheme \eqref{dgEuler} in the 2D setting can produce negative cell average values at certain times  $t_0 \in (0,T)$, thereby indicating a failure of the non-modified algorithm.

\begin{table}[!htb]
\caption{The minimum cell averages of $c_1$ and $c_2$ at $t=t_0$ with meshes $N\times N$. }
\begin{tabular}{ |c|c|c|c| }
\hline
 $N$ & $t_0$ & $c_1$ & $c_2$ \\ 
\hline
 10  &   0.0052381   & -0.00019494   &  -6.8271e-05 \\
\hline
 20  &   0.0010976   & -3.3455e-05    &  1.318e-05 \\
\hline
 30  &   0.0004918   & -1.2657e-05    &  -5.4982e-06 \\
\hline
 40  &   0.00027692   & -6.2495e-06    & 3.873e-06\\
\hline
\end{tabular}
\label{tab:c12cellavg}
\end{table}

Next, we investigate the DG scheme \eqref{fullyMD1}-\eqref{fullyMD3}, using the numerical fluxes \eqref{betaMD} selectively when \eqref{dgEuler} produces negative cell averages. To ensure positive cell averages, we adjust the time step to satisfy \eqref{CFL1D+} in the 2D setting, where the parameter $\gamma$ depends on the total number of the quadrature points in each direction. 

We calculate the $l_1$  errors and orders of convergence at $t=0.01$ using various numbers of Gauss-Lobatto quadrature points and different values of  $\gamma$. 
The results are summarized in \Cref{tab:ex2DmGL}. Notably, we observe third-order accuracy and find that as the total number of quadrature points increases, 
$\gamma$ can be set to a larger value. This adjustment improves the approximation of the parameter $\lambda$ in   \eqref{CFL1D} or \eqref{CFL1D+}, thereby 
enhancing numerical stability.

\begin{table}[!htb]
\caption{$l^1$ errors and orders at $t=0.01$ with meshes $N\times N$, $m$ Gauss-Lobatto quadrature points in each direction, and the parameter $\gamma$ in \eqref{CFL1D+}. }
\begin{tabular}{ |c|l|c|c| c|c| c|c| }
\hline
$(m,\gamma)$& $N$ & $c_1$ error & order & $c_2$ error & order & $\psi$ error & order \\ \hline
\multirow{4}{*}{$(3,{1}/{4})$}
 & 10 & 3.65605e-03 &  --  & 1.73140e-03 &  --  & 2.48497e-04 &  --   \\
 & 20 & 6.15627e-05 & 5.89 & 1.23012e-05 & 7.14 & 7.67527e-06 & 5.02  \\
 & 30 & 4.70311e-06 & 6.34 & 2.34902e-06 & 4.08 & 1.03025e-06 & 4.95  \\
 & 40 & 2.51946e-06 & -- & 1.29628e-06 & -- & 4.09384e-07 & 3.21  \\
\hline
\multirow{4}{*}{$(4,{1}/{2})$}
 & 10 & 7.73601e-05 &  --  & 3.49995e-05 &  --  & 2.69952e-05 &  --   \\
 & 20 & 4.71377e-06 & 4.04 & 2.26200e-06 & 3.95 & 3.15719e-06 & 3.10  \\
 & 30 & 1.06658e-06 & 3.67 & 5.29898e-07 & 3.58 & 9.23417e-07 & 3.03  \\
 & 40 & 4.29086e-07 & 3.17 & 2.13824e-07 & 3.15 & 3.87547e-07 & 3.02  \\
\hline
 \multirow{4}{*}{$(5,{1})$}
 & 10 & 1.26670e-04 &  --  & 4.12020e-05 &  --  & 2.48704e-05 &  --   \\
 & 20 & 4.35146e-06 & 4.86 & 2.09879e-06 & 4.30 & 2.68506e-06 & 3.21  \\
 & 30 & 1.08331e-06 & 3.43 & 5.37169e-07 & 3.36 & 7.71977e-07 & 3.07  \\
 & 40 & 4.37280e-07 & 3.15 & 2.18748e-07 & 3.12 & 3.21122e-07 & 3.05  \\
 \hline
\end{tabular}
\label{tab:ex2DmGL}
\end{table}

\subsection{2D tests on solution properties} 
 \label{sec:2d} In this 2D example, we test solution properties, including solution positivity, mass conservation, and free energy dissipation.  
Following a methodology similar to \cite[Example 4]{LWYY22}, we assess the effectiveness of the proposed scheme in solving  the two-dimensional PNP system \eqref{PNP} with $m=2$ in $\Omega=[0,1]^2$,
given by 
\begin{align*}
&\partial_t c_1= \nabla\cdot(\nabla c_1+c_1\nabla\psi),\\
&\partial_t c_2= \nabla\cdot(\nabla c_2-c_2\nabla\psi),\\
&-\Delta \psi =  c_1-c_2,\\
& c_1^{\rm in}(x,y)=\frac{1}{2}x^2(1-x)^2(1-\cos(\pi y)), \\
& c_2^{\rm in}(x,y)=\pi \sin(\pi x)y^2(1-y)^2,\\
& \frac{\partial c_i}{\partial  \textbf{n}}+q_ic_i \frac{\partial \psi}{\partial  \textbf{n}} =0 ,\quad (x,y) \in \partial \Omega,\\
&\psi =0 \mbox{~~on~} \partial\Omega_D,  \mbox{~~and~} \frac{\partial \psi}{\partial  \textbf{n}}  =0  \mbox{~~on~} \partial\Omega_N, \quad t>0,
\end{align*}
where $\partial \Omega_D=\{(x,y)\in \bar{\Omega}: x=0, x=1 \}$ and $\partial \Omega_N=\partial \Omega \backslash \partial \Omega_D$, with ion charges $q_1 =1$ and $q_2=-1$.

For this example, the DG schemes use $P^2$ polynomials with numerical flux parameters $\beta_0=16$ and $\beta_1=\frac{1}{6}$. Following the approach in  \cite[Example 4]{LWYY22}, we set the initial time step to  $\Delta t= 10^{-5}$ and use a mesh size of $20\times 20$. 

{\bf \noindent Test case 1.} We first solve this problem using the DG scheme \eqref{dgEuler} in the 2D setting. By $t=4\times 10^{-5}$, the minimum cell average of $c_1$ is $-1.8946\times 10^{-6}$, indicating a failure of the scheme.

{\bf \noindent Test case 2.} 
Subsequently, we implement the DG scheme \eqref{fullyMD1}-\eqref{fullyMD3}, employing the numerical fluxes \eqref{betaMD} selectively at time steps where \eqref{dgEuler} produces negative cell averages. 
To address solution positivity, Figure \ref{ex4cg}(a) and \ref{ex4cg}(b) show the evolution of the smallest cell averages of $c_1, c_2$ for $t\in [0,1]$. Notably, from Figure \ref{ex4cg}, it is clear that the smallest cell averages of both $c_1$ and $c_2$ remain positive for all time $t\in [0, 1]$.

\begin{figure}
\centering
\subfigure[Smallest cell average of $c_1$.]{\includegraphics[width=0.49\textwidth]{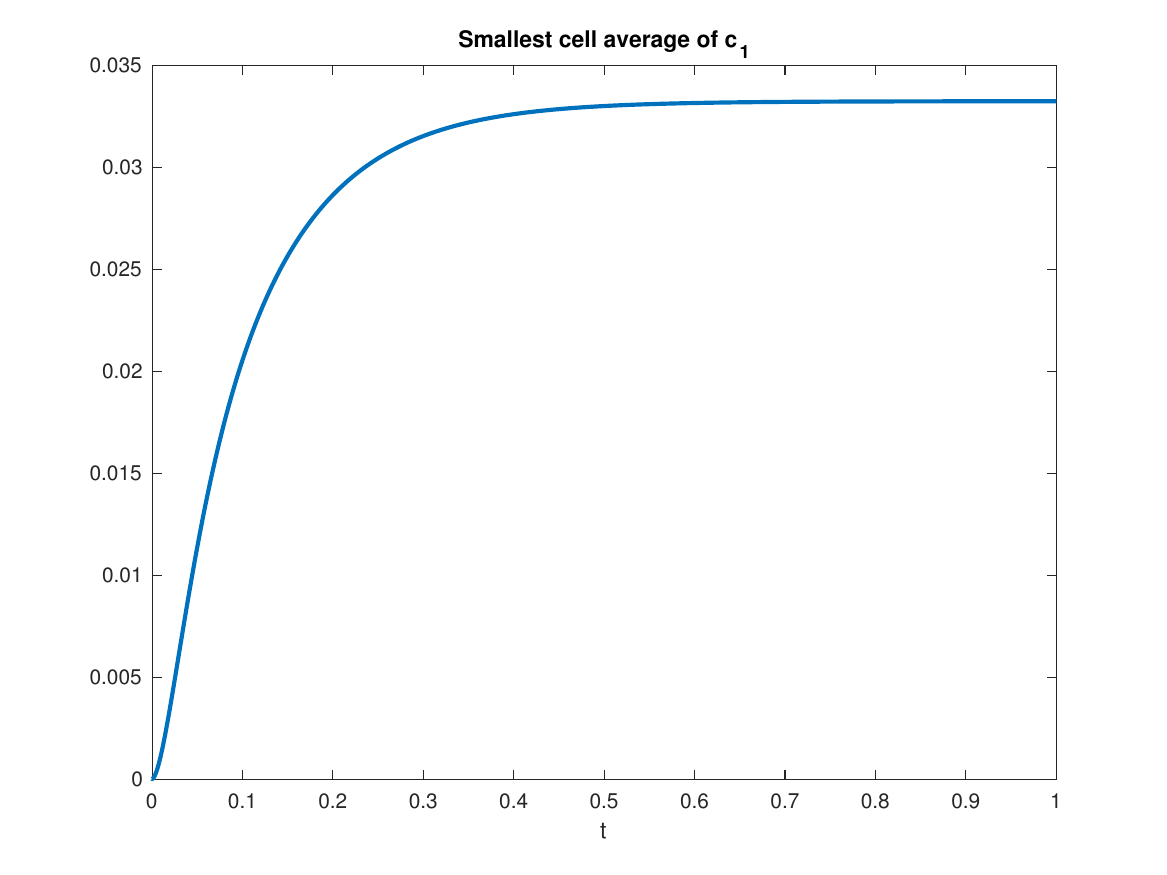}}
\subfigure[Smallest cell average of $c_2$.]{\includegraphics[width=0.49\textwidth]{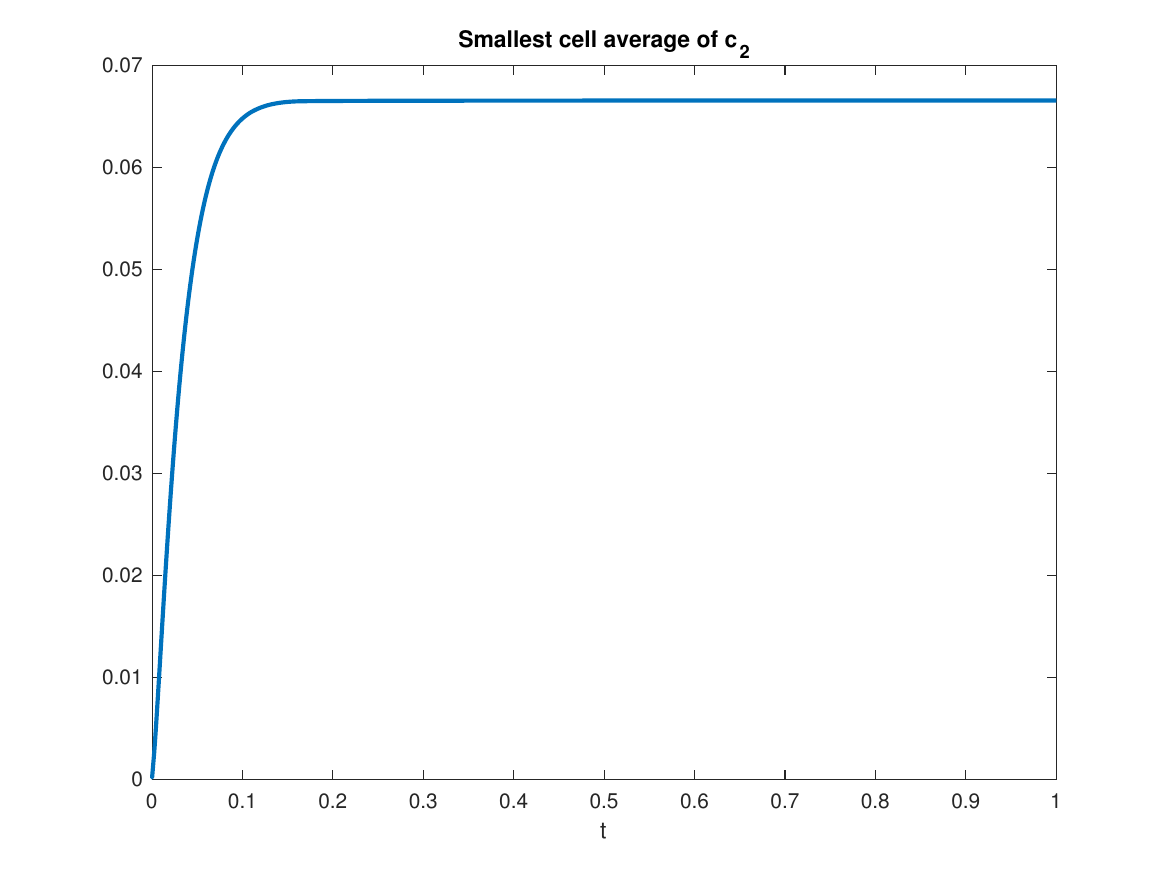}}
\caption{ Smallest cell averages of $c_1, c_2$.
  } \label{ex4cg}
 \end{figure}

{\bf \noindent Test case 3.} 
Regarding the conservation of mass and the free energy dissipation,  Figure \ref{ex4em} illustrates the total mass evolution for both $c_1$ and $c_2$. The results show that the DG scheme conserves the total mass, with $\frac{1}{30}$ for $c_1$ and $\frac{2}{30}$ for $c_2$. Furthermore, Figure \ref{ex4em} presents the evolution of the energy, confirming the dissipation of free energy over time.

 \begin{figure}[!htb]
\caption{Conservation of mass and decay of free energy}
\centering
\begin{tabular}{cc}
\includegraphics[width=\textwidth]{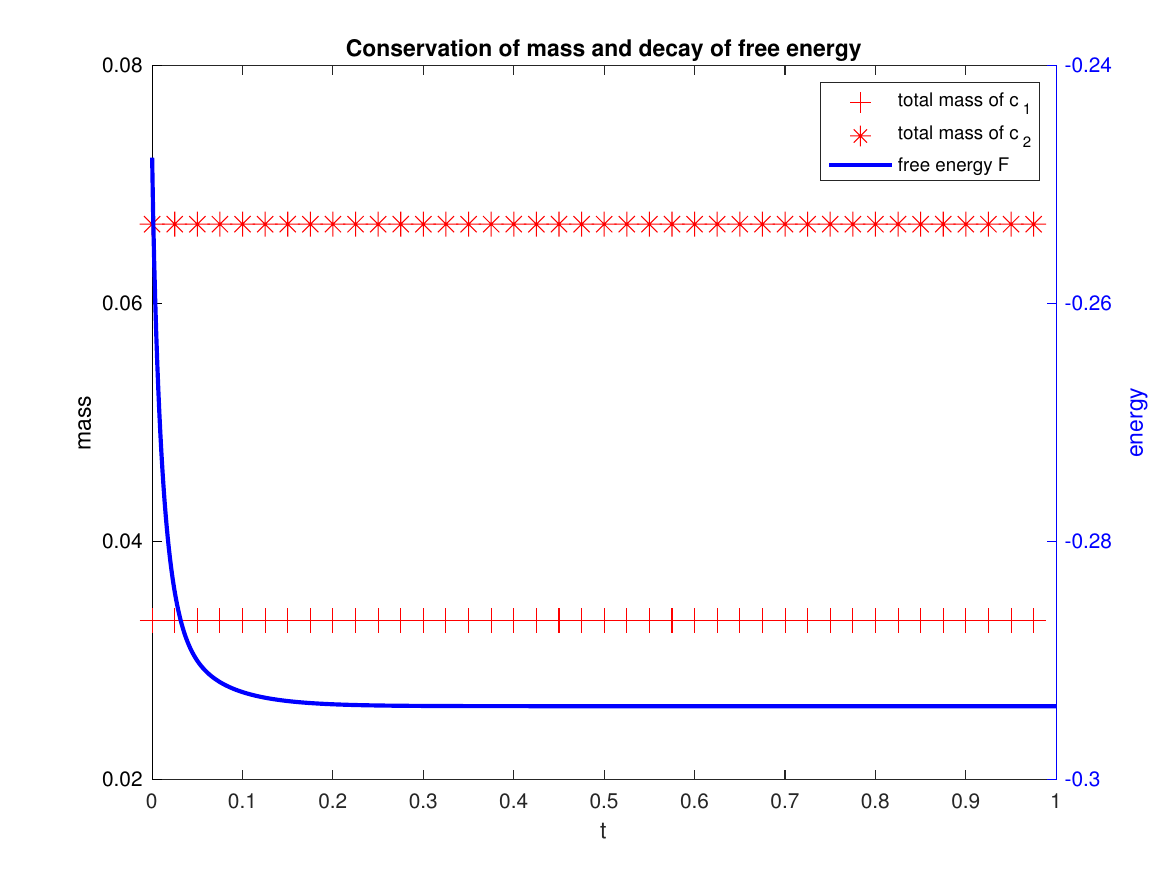} 
\end{tabular}
\label{ex4em}
\end{figure}

Finally, we analyze the evolution of $c_1, c_2$,  and $\psi$ for $t\in (0,1]$.  Figure \ref{ex5pattern} shows the contours of $c_1 -1/30$ (first column), $c_2-2/30$ (second column) and $\psi$ (third column) at $t=0, 0.2, 0.5$, and $t=1$.  Notably, the contours of $c_2$ and $\psi$ at $t=0.5$ and $t=1.0$ are nearly identical, suggesting that the solution approaches a steady state by $t=1$.

\begin{figure}
\centering
\subfigure{\includegraphics[width=0.325\textwidth]{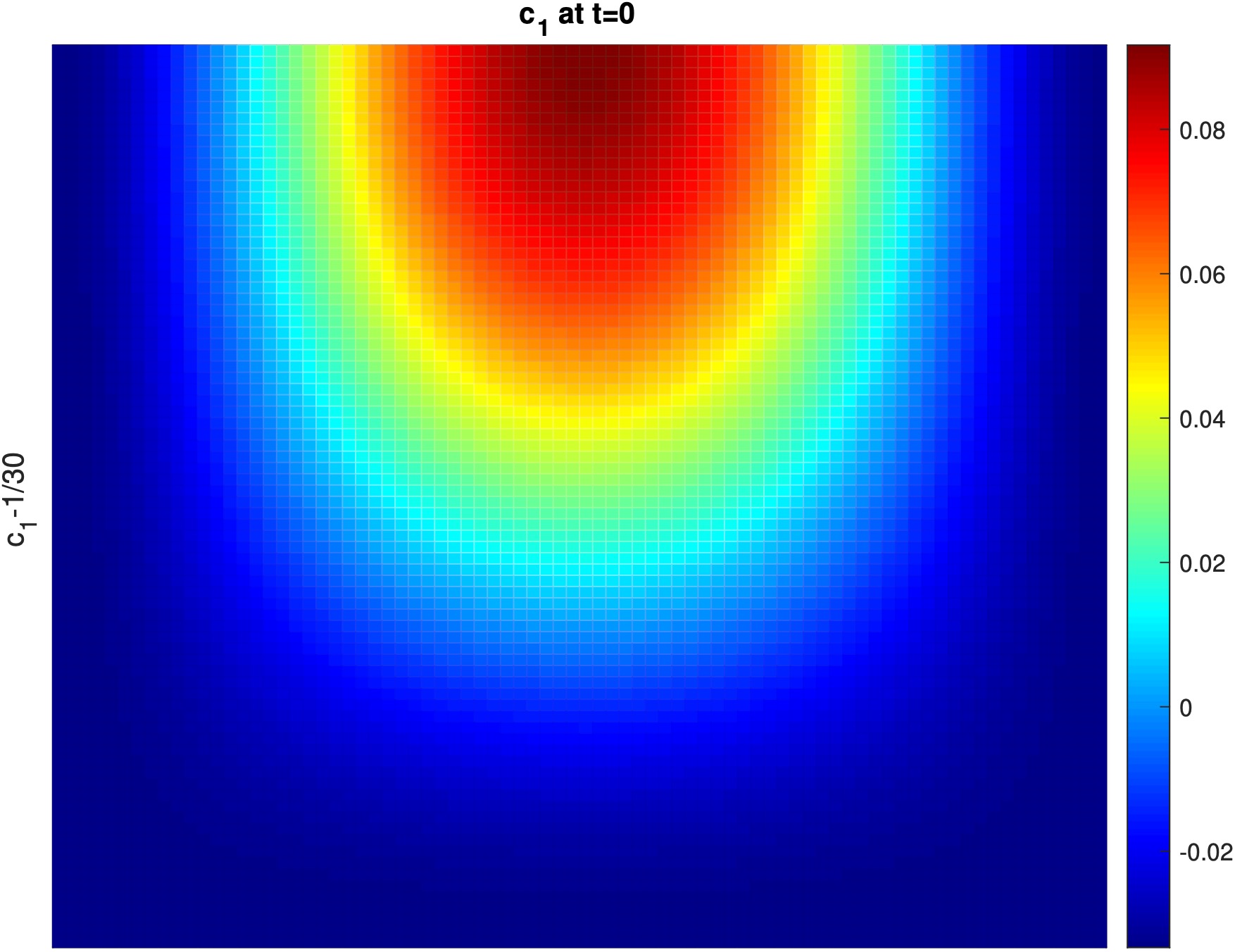}}
\subfigure{\includegraphics[width=0.325\textwidth]{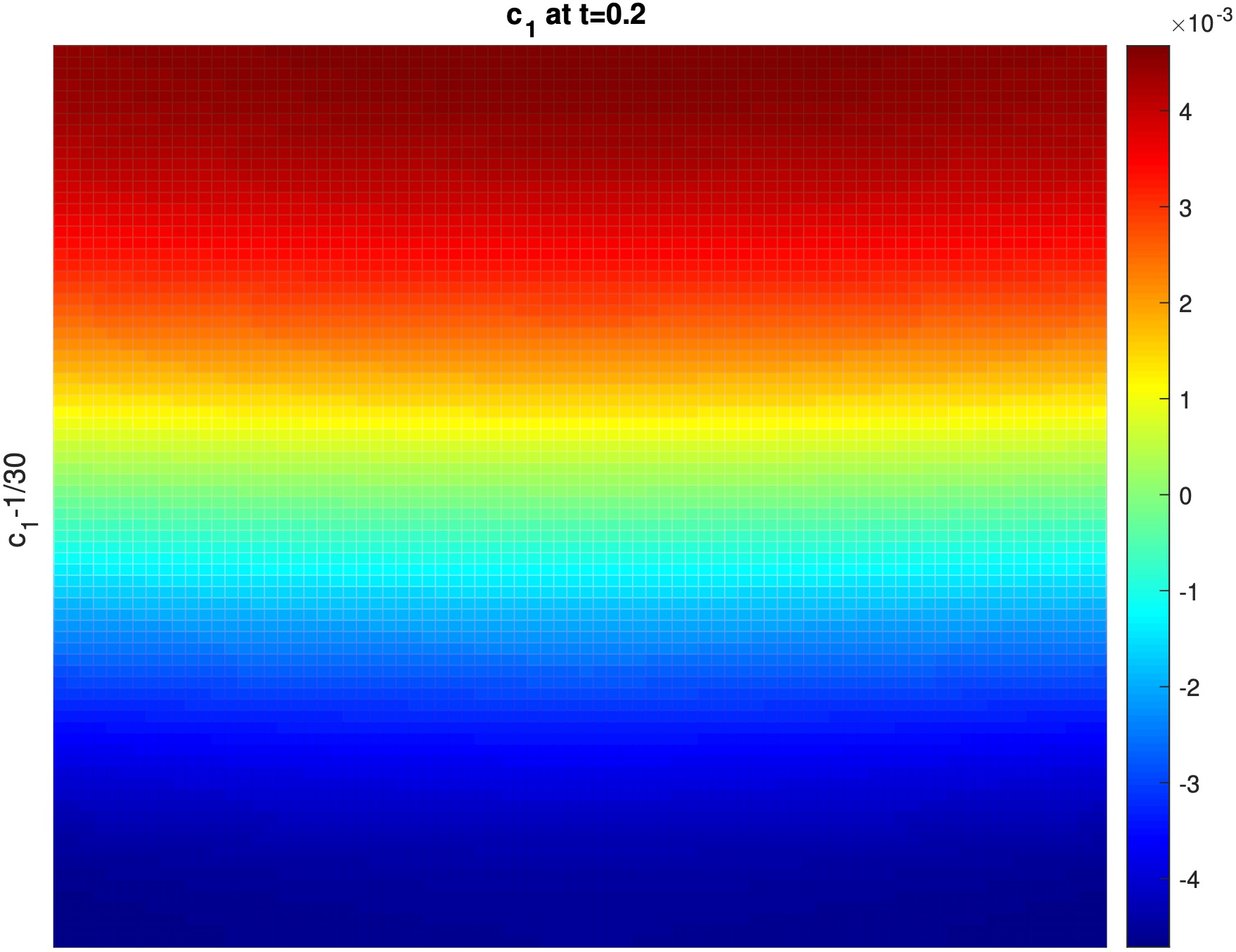}}
\subfigure{\includegraphics[width=0.325\textwidth]{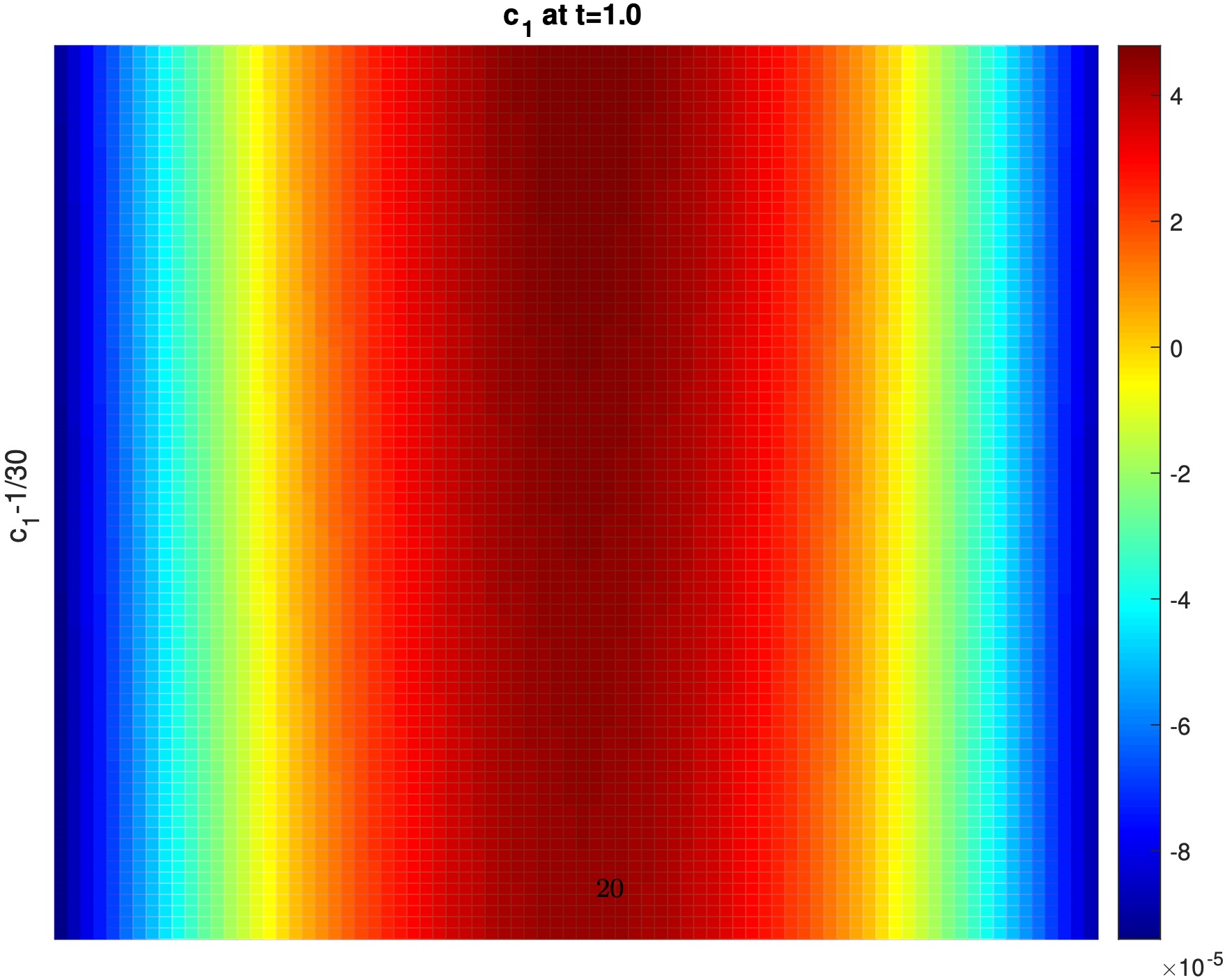}}
\subfigure{\includegraphics[width=0.325\textwidth]{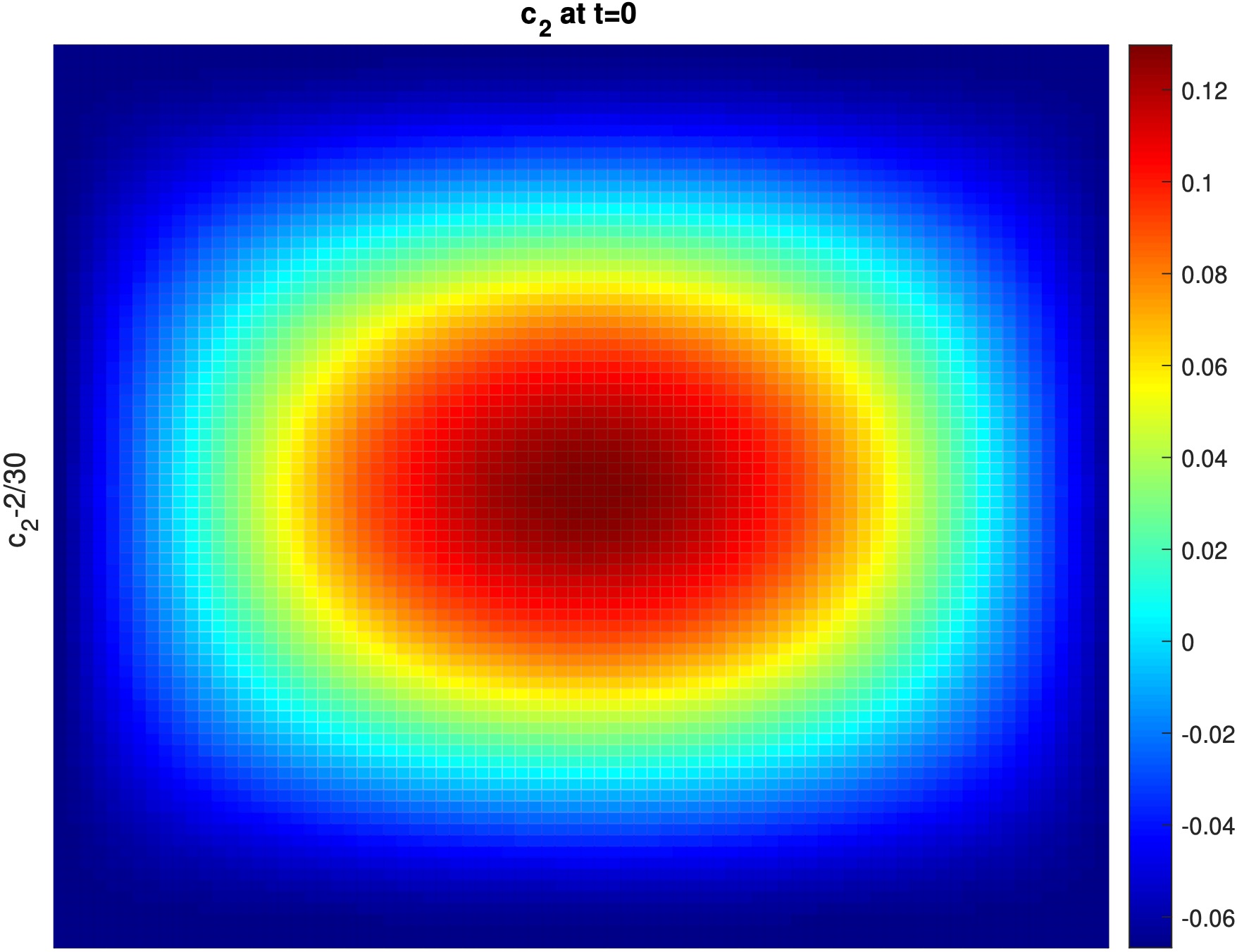}}
\subfigure{\includegraphics[width=0.325\textwidth]{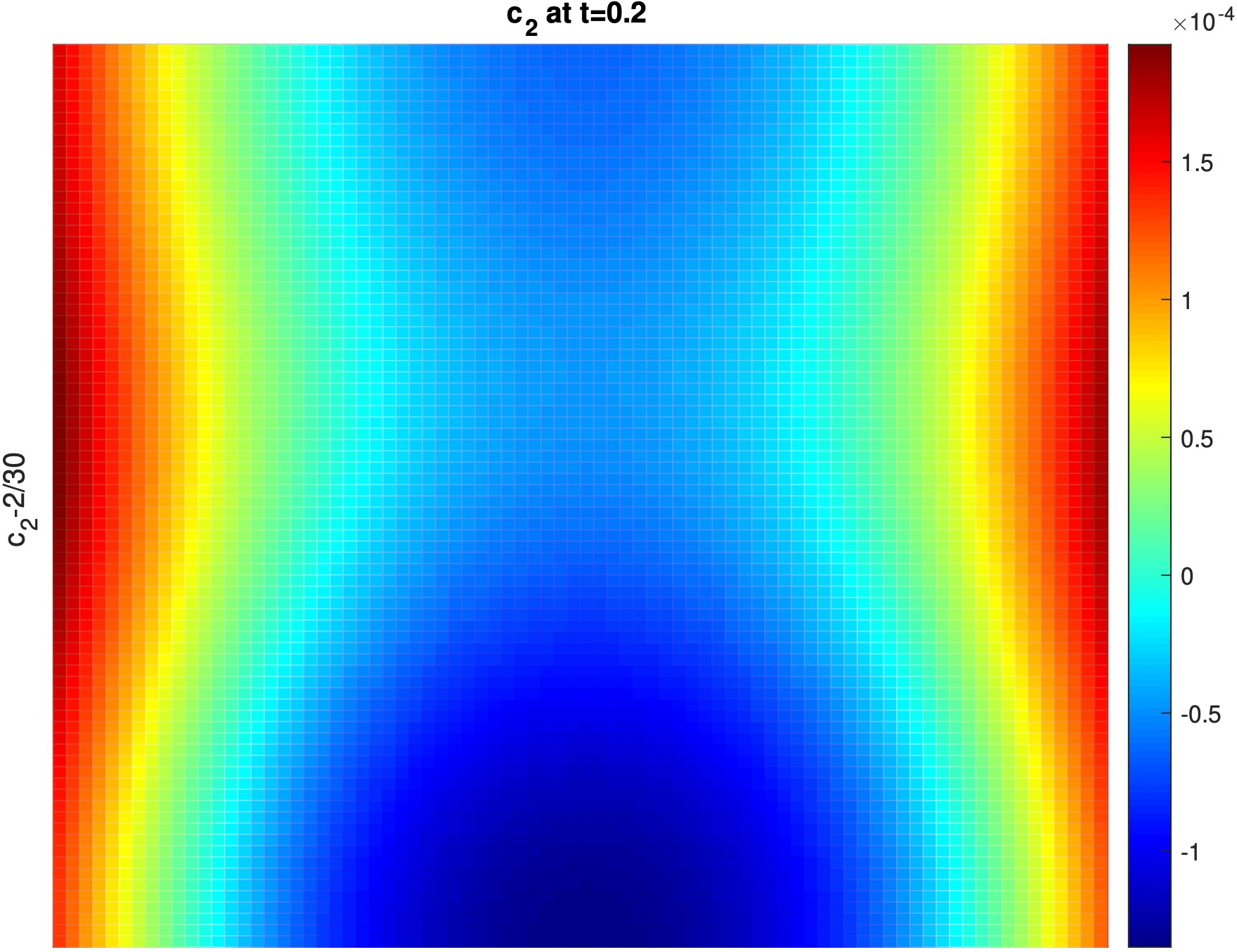}}
\subfigure{\includegraphics[width=0.325\textwidth]{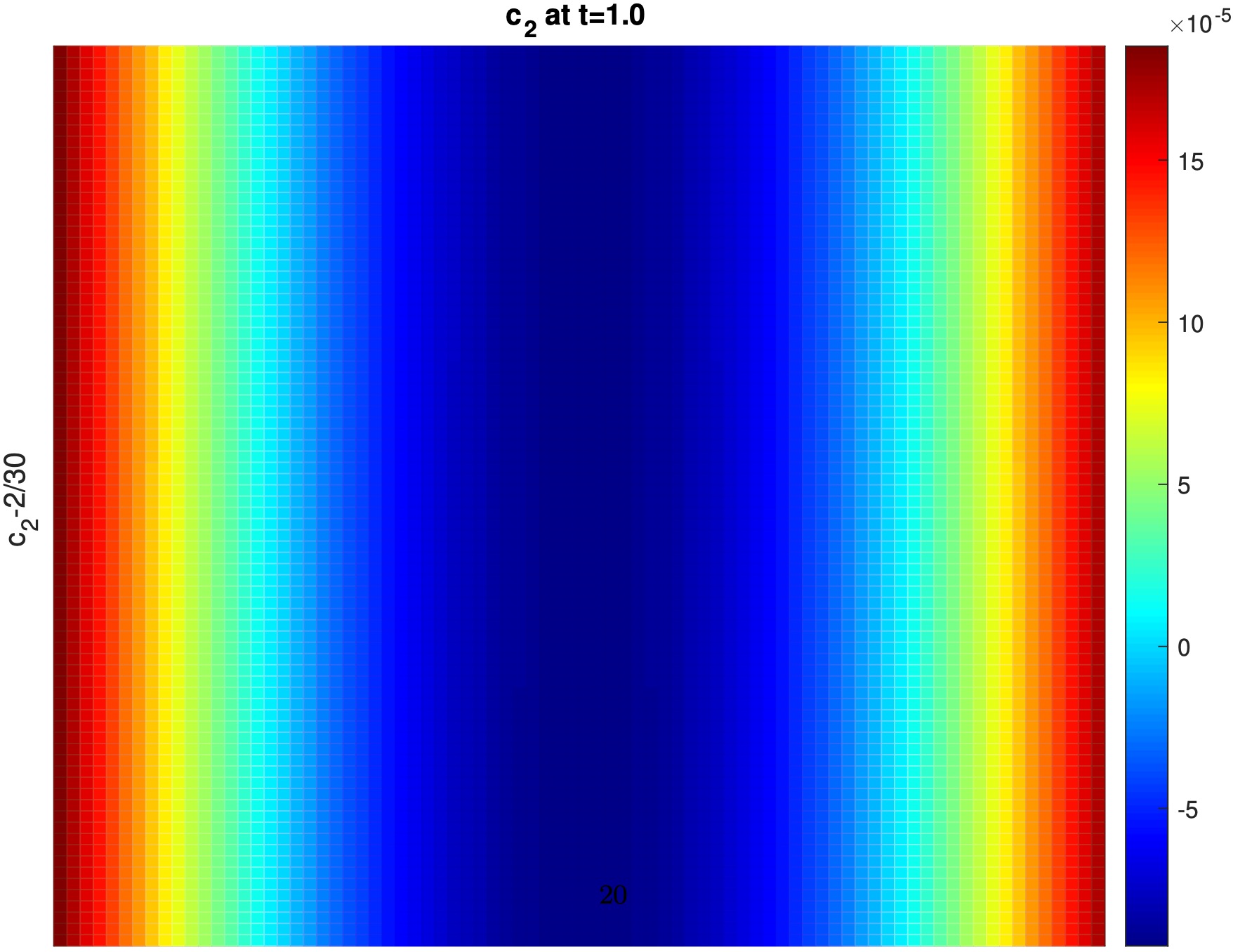}}
\subfigure{\includegraphics[width=0.325\textwidth]{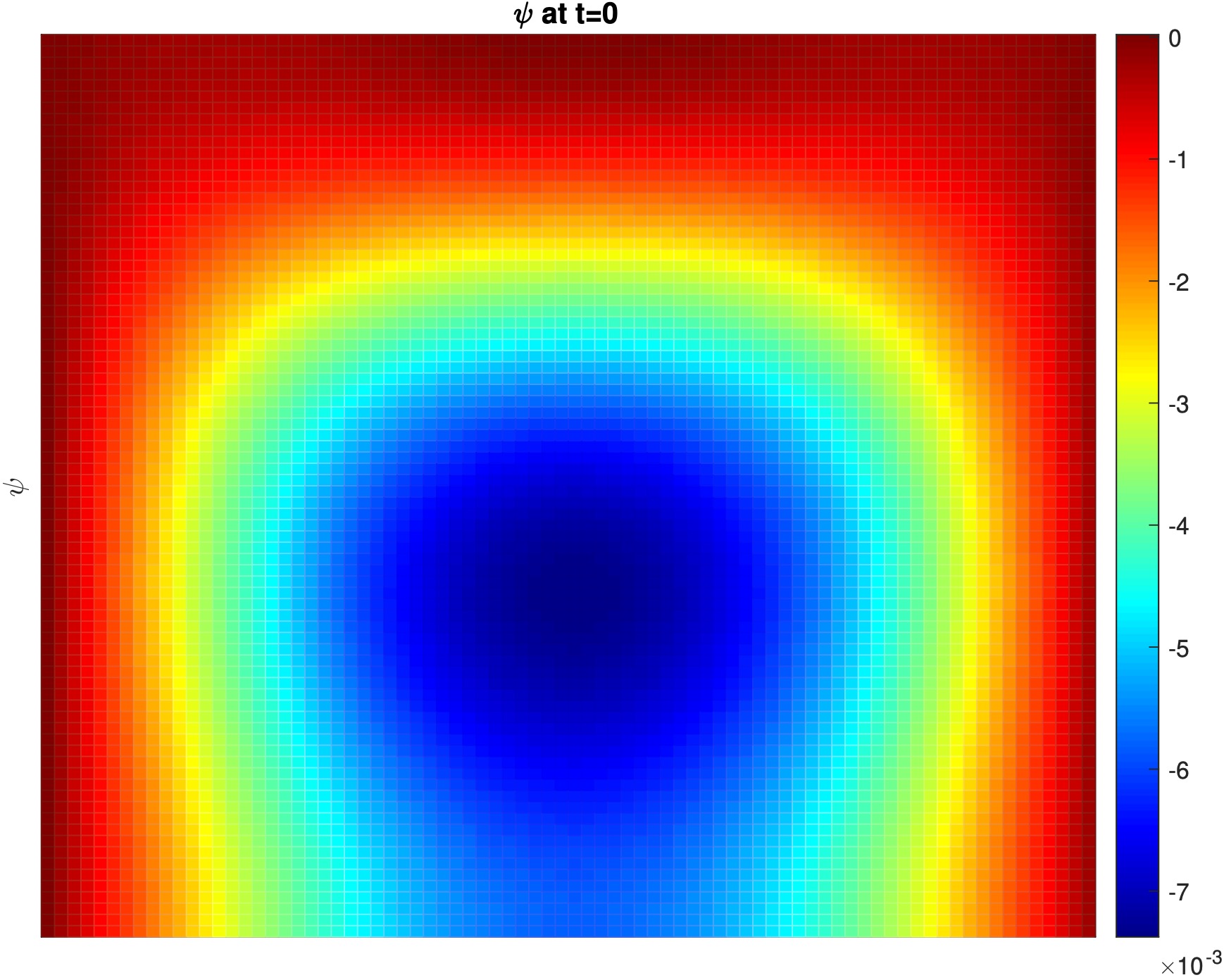}}
\subfigure{\includegraphics[width=0.325\textwidth]{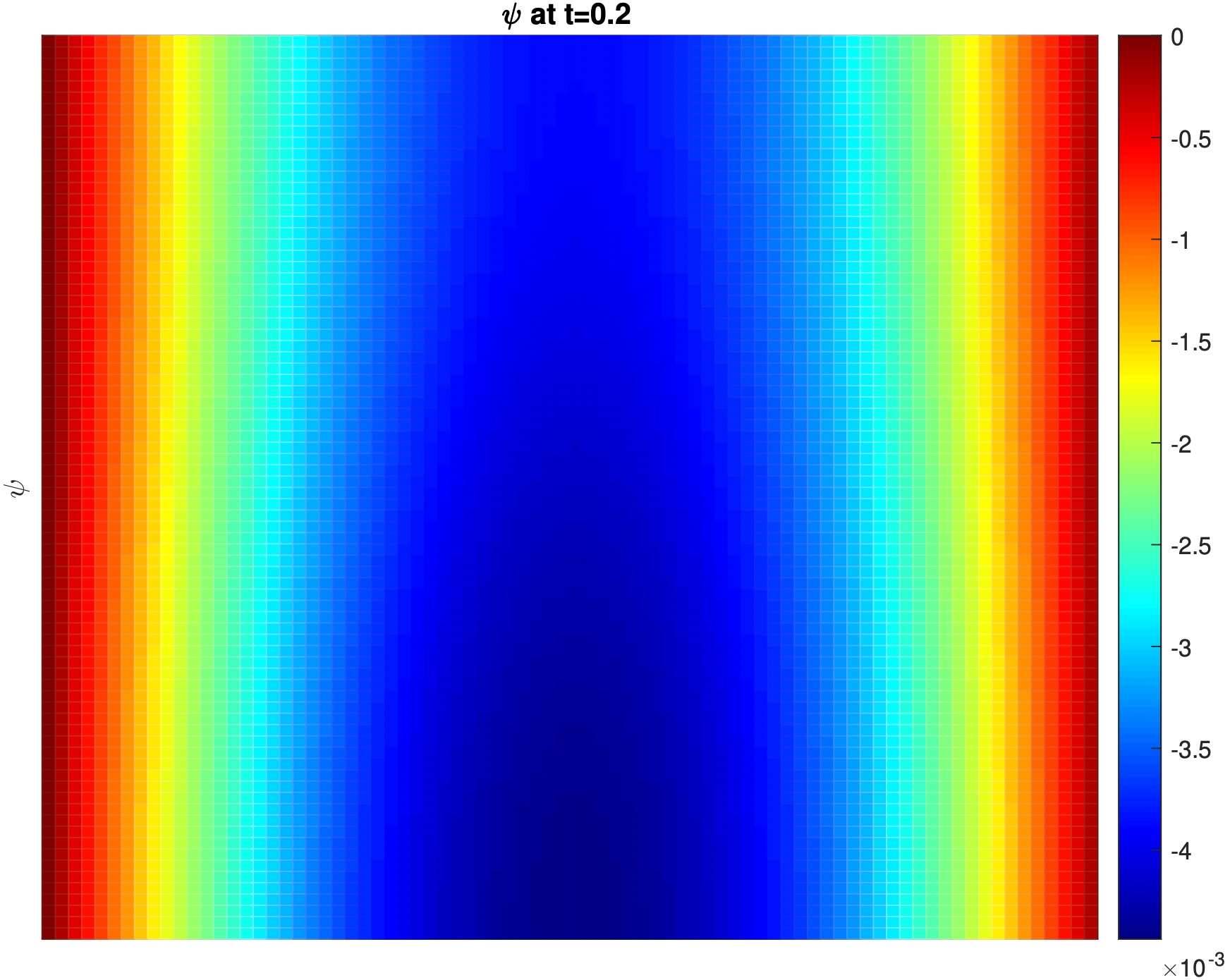}}
\subfigure{\includegraphics[width=0.325\textwidth]{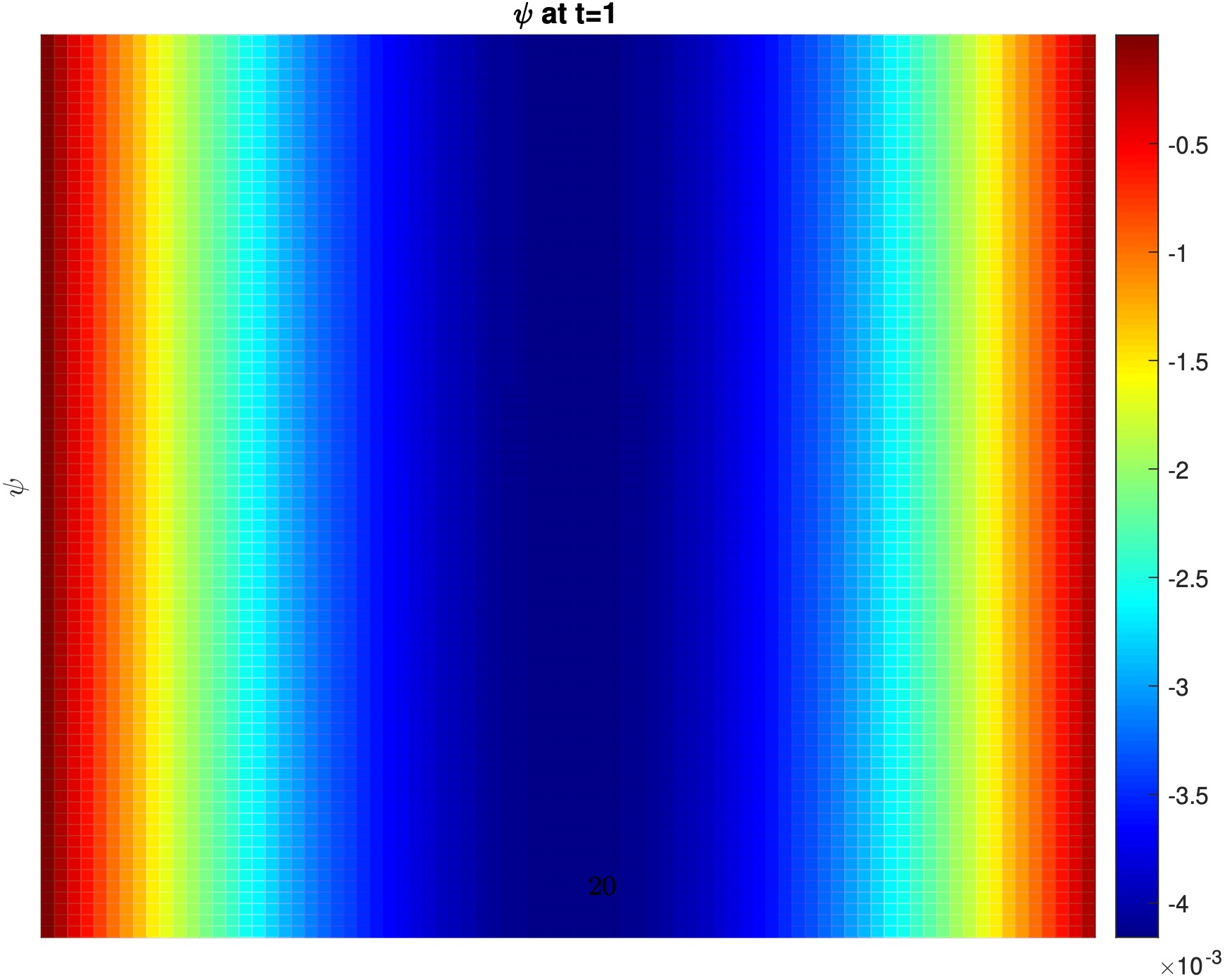}}
\caption{ The contours evolution of $c_1-1/30$, $c_2-2/30$ and $\psi$. } \label{ex5pattern}
\end{figure}

\section{Concluding remarks}

In this paper, we develop a novel positivity-preserving numerical flux to replace the DDG flux introduced in \cite{LW17} for the PNP equations.   
By employing Gauss-Lobatto quadrature in the new flux formulation, we establish that the scheme preserves positive cell-averaged concentrations under a forward Euler time discretization, provided a proper CFL condition is satisfied. To ensure point-wise positivity, which is crucial for the validity of our logarithmic reformulations, we apply a subsequent positivity-preserving reconstruction. 
This hybrid algorithm fulfills the goal of constructing  arbitrarily high-order (in space) positivity-preserving DDG schemes for the PNP equations.  

Our extensive numerical experiments 
in both one and two-dimensional settings
also demonstrate mass conservation and energy dissipation.
 However, due to the coupling between the positivity-preserving numerical flux and the reconstruction limiter, rigorously proving discrete energy dissipation remains a challenging task and is left for future work.

\bibliographystyle{abbrv}
\bibliography{PNP}

\begin{thebibliography}{10}

\bibitem{BB00}
J.-D. Benamou and Y.~Brenier.
\newblock A computational fluid mechanics solution to the {M}onge-{K}antorovich mass transfer problem.
\newblock {\em Numer. Math.}, 84(3):375--393, 2000.

\bibitem{CCH15}
J.~A. Carrillo, A.~Chertock, and Y.~H. Huang.
\newblock A finite-volume method for nonlinear nonlocal equations with a gradient flow structure.
\newblock {\em Commun. Comput. Phys.}, 17(1):233--258, 2015.

\bibitem{CLY25}
J.~A. Carrillo, H.~Liu, and H.~Yu.
\newblock Positivity-preserving and energy-dissipating discontinuous {G}alerkin methods for nonlinear nonlocal {F}okker-{P}lanck equations.
\newblock {\em Commun. Appl. Ind. Math.}, 2025.

\bibitem{CLX25}
H.~Chen, H.~Liu, and X.~Xu.
\newblock The {O}nsager principle and structure preserving numerical schemes.
\newblock {\em J. Comput. Phys.}, 523:113679, 2025.

\bibitem{Da97}
S.~Datta.
\newblock {\em Electronic Transport in Mesoscopic Systems}.
\newblock Cambridge University Press, 1997.

\bibitem{DingWangZhou_JCP2019}
J.~Ding, Z.~Wang, and S.~Zhou.
\newblock Positivity preserving finite difference methods for {P}oisson-{N}ernst-{P}lanck equations with steric interactions: application to slit-shaped nanopore conductance.
\newblock {\em J. Comput. Phys.}, 397:108864, 2019.

\bibitem{DingWangZhou_JCP2020}
J.~Ding, Z.~Wang, and S.~Zhou.
\newblock Structure-preserving and efficient numerical methods for ion transport.
\newblock {\em J. Comput. Phys.}, 418:109597, 2020.

\bibitem{EL07}
B.~Eisenberg and W.~Liu.
\newblock Poisson-{N}ernst-{P}lanck systems for ion channels with permanent charges.
\newblock {\em SIAM J. Math. Anal.}, 38:1932--1966, 2007.

\bibitem{Gl42}
S.~Glasstone.
\newblock {\em An introduction to Electrochemstry}.
\newblock Van Nostrand Company, Inc., Princeton, NJ, 1942.

\bibitem{GST2001}
S.~Gottlieb, C.~Shu, and E.~Tadmor.
\newblock Strong stability-preserving high-order time discretization methods.
\newblock {\em SIAM Rev.}, 43(1):89--112, 2001.

\bibitem{HPY19}
D.~He, K.~Pan, and X.~Yue.
\newblock A positivity preserving and free energy dissipative difference scheme for the {P}oisson-{N}ernst-{P}lanck system.
\newblock {\em J. Sci. Comput.}, 81:436--458, 2019.

\bibitem{Hi01}
B.~Hille.
\newblock {\em Ion channels of excitable membranes}, volume 507.
\newblock Sinauer Sunderland, MA, 2001.

\bibitem{HH2020}
J.~Hu and X.~Huang.
\newblock A fully discrete positivity-preserving and energy-dissipative finite difference scheme for {P}oisson-{N}ernst-{P}lanck equations.
\newblock {\em Numer. Math.}, 145:77--115, 2020.

\bibitem{Je96}
J.~Jerome.
\newblock {\em Analysis of Charge Transport: A Mathematical Study of Semiconductor Devices}.
\newblock Springer, Berlin, 1996.

\bibitem{JKO98}
R.~Jordan, D.~Kinderlehrer, and F.~Otto.
\newblock The variational formulation of the {F}okker--{P}lanck equation.
\newblock {\em SIAM J. Math. Anal.}, 29(1):1--17, 1998.

\bibitem{KLX17}
D.~Kinderlehrer, L.~Monsaingeon, and X.~Xu.
\newblock A wasserstein gradient flow approach to {P}oisson-{N}ernst-{P}lanck equations.
\newblock {\em ESAIM Control Optim. Calc. Var.}, 23(1):137--164, 2017.

\bibitem{Li04}
D.~Li.
\newblock {\em Electrokinetics in Microfluidics}.
\newblock Academic Press, 2004.

\bibitem{LiuC2021a}
C.~Liu, C.~Wang, S.~M. Wise, X.~Yue, and S.~Zhou.
\newblock A positivity-preserving, energy stable and convergent numerical scheme for the {Poisson-Nernst-Planck} system.
\newblock {\em Math. Comp.}, 90:2071--2106, 2021.

\bibitem{LWWYZ2023}
C.~Liu, C.~Wang, S.~M. Wise, X.~Yue, and S.~Zhou.
\newblock A second order accurate, positivity preserving numerical method for the {P}oisson-{N}ernst-{P}lanck system and its convergence analysis.
\newblock {\em J. Sci. Comput.}, 97(23), 2023.

\bibitem{Liu21}
H.~Liu.
\newblock Analysis of direct discontinuous {G}alerkin methods for multi-dimensional convection--diffusion equations.
\newblock {\em Numer. Math.}, 147(4):839--867, 2021.

\bibitem{LM20b}
H.~Liu and W.~Maimaitiyiming.
\newblock Unconditional positivity-preserving and energy stable schemes for a reduced {P}oisson-{N}ernst-{P}lanck system.
\newblock {\em Commun. Comput. Phys.}, 27:1505--1529, 2020.

\bibitem{LM21}
H.~Liu and W.~Maimaitiyiming.
\newblock Efficient, positive, and energy stable schemes for multi-{D} {P}oisson-{N}ernst-{P}lanck systems.
\newblock {\em J. Sci. Comput.}, 87:92, 2021.

\bibitem{LM23}
H.~Liu and W.~Maimaitiyiming.
\newblock A dynamic mass transport method for {P}oisson-{N}ernst-{P}lanck equations.
\newblock {\em J. Comput. Phys}, 473:111699, 2023.

\bibitem{LW17}
H.~Liu and Z.~Wang.
\newblock A free energy satisfying discontinuous {G}alerkin method for one-dimensional {P}oisson-{N}ernst-{P}lanck systems.
\newblock {\em J. Comput. Phys.}, 328:413--437, 2017.

\bibitem{LWYY22}
H.~Liu, Z.~Wang, P.~Yin, and H.~Yu.
\newblock Positivity-preserving third order {DG} schemes for {P}oisson-{N}ernst-{P}lanck equations.
\newblock {\em J. Comput. Phys.}, 452:110777, 2022.

\bibitem{LY09}
H.~Liu and J.~Yan.
\newblock The direct discontinuous {G}alerkin ({DDG}) methods for diffusion problems.
\newblock {\em SIAM J. Numer. Anal}, 47:675--698, 2009.

\bibitem{LY10}
H.~Liu and J.~Yan.
\newblock The direct discontinuous {G}alerkin ({DDG}) method for diffusion with interface corrections.
\newblock {\em Commun. Comput. Phys.}, 8(3):541--564, 2010.

\bibitem{LY15}
H.~Liu and H.~Yu.
\newblock The entropy satisfying discontinuous {G}alerkin method for {F}okker--{P}lanck equations.
\newblock {\em J. Sci. Comput.}, 62:803--830, 2015.

\bibitem{Ma86}
P.~Markowich.
\newblock {\em The Stationary Semiconductor Device Equations}.
\newblock Springer-Verlag, New York, 1986.

\bibitem{MRS90}
P.~A. Markowich, C.~A. Ringhofer, and C.~Schmeiser.
\newblock {\em Semiconductor Equations}.
\newblock Springer-Verlag Inc, New York, 1990.

\bibitem{MXL16}
M.~Metti, J.~Xu, and C.~Liu.
\newblock Energetically stable discretizations for charge transport and electrokinetic models.
\newblock {\em J. Comput. Phys.}, 306:1--18, 2016.

\bibitem{Mo83}
M.~S. Mock.
\newblock {\em Analysis of Mathematical Models of Semiconductor Devices}, volume~3.
\newblock Boole Press, 1983.

\bibitem{ShenXu_NM21}
J.~Shen and J.~Xu.
\newblock Unconditionally positivity preserving and energy dissipative schemes for {Poisson--Nernst--Planck} equations.
\newblock {\em Numer. Math.}, 148:671--697, 2021.

\bibitem{ST22}
S.~Su and H.~Tang.
\newblock A positivity-preserving and free energy dissipative hybrid scheme for the poisson-nernst-planck equations on polygonal and polyhedral meshes.
\newblock {\em Comput. Math. Appl.}, 108:33--48, 2022.

\bibitem{SCS18}
Z.~Sun, J.~A. Carrillo, and C.-W. Shu.
\newblock A discontinuous galerkin method for nonlinear parabolic equations and gradient flow problems with interaction potentials.
\newblock {\em J. Comp. Phys.}, 352:76--104, 2018.

\bibitem{YHL14}
P.~Yin, Y.~Huang, and H.~Liu.
\newblock An iterative discontinuous {G}alerkin method for solving the nonlinear {P}oisson-{B}oltzmann equation.
\newblock {\em Commun. Comput. Phys.}, 16(2):491--515, 2014.

\bibitem{YHL18}
P.~Yin, Y.~Huang, and H.~Liu.
\newblock Error estimates for the iterative discontinuous {G}alerkin method to the nonlinear {P}oisson-{B}oltzmann equation.
\newblock {\em Commun. Comput. Phys.}, 23(1):168--197, 2018.

\bibitem{Zh17}
X.~Zhang.
\newblock On positivity-preserving high order discontinuous galerkin schemes for compressible {N}avier-{S}tokes equations.
\newblock {\em J. Comput. Phys.}, 328:301--343, 2017.

\bibitem{ZS10}
X.~Zhang and C.-W. Shu.
\newblock On maximum-principle-satisfying high order schemes for scalar conservation laws.
\newblock {\em J. Comput. Phys.}, 229:3091--3120, 2010.

\end{thebibliography}

\end{document}